\documentclass{article}
\usepackage{amsmath}
\usepackage{amssymb}
\usepackage{color}
\usepackage{graphicx}
\usepackage{xr}
\usepackage{fullpage}
\usepackage{here}

\usepackage{epsfig,psfrag}
\usepackage[space]{grffile}
\usepackage{mathtools}
\usepackage{enumerate}
\usepackage{etoolbox}
\usepackage{bbm}
\usepackage{mhchem}
\usepackage{mathrsfs}
\usepackage{dsfont,accents}

\DeclareMathOperator*{\diag}{diag}

\DeclareMathOperator{\eps}{\varepsilon}

\def\llongrightarrow{\relbar\joinrel\relbar\joinrel\relbar\joinrel\rightarrow}

\providecommand{\rarrow}[1]{\stackrel{#1}{\llongrightarrow}}

\providecommand\X[1]{\boldsymbol{X_{#1}}}
\providecommand\R[1]{\boldsymbol{R_{#1}}}

\providecommand\phib{\boldsymbol{\emptyset}}

\providecommand\V[1]{\boldsymbol{V_{#1}}}
\providecommand\Y[1]{\boldsymbol{Y_{#1}}}
\let\P\relax
\providecommand\P[1]{\boldsymbol{P_{#1}}}
\providecommand\M[1]{\boldsymbol{M_{#1}}}
\providecommand\Q[1]{\boldsymbol{Q_{#1}}}

\newenvironment{proof}{{\it Proof :~}}{\hfill$\diamondsuit$\\}

\newtheorem{theorem}{Theorem}

\newtheorem{proposition}[theorem]{Proposition}

\newtheorem{example}[theorem]{Example}
\newtheorem{hyp}[theorem]{Assumption}

\newtheorem{remark}[theorem]{Remark}

\newtheorem{problem}[theorem]{Problem}

\providecommand{\blue}[1]{\textcolor[rgb]{0,0,0}{#1}}

\def\llongrightarrow{\relbar\joinrel\relbar\joinrel\relbar\joinrel\rightarrow}

\providecommand{\rarrow}[1]{\stackrel{#1}{\llongrightarrow}}

\def\Xz{\boldsymbol{X}}
\def\Vz{\boldsymbol{V}}
\def\Rz{\boldsymbol{R}}
\def\Yz{\boldsymbol{Y}}
\def\Uz{\boldsymbol{U}}
\def\Sz{\boldsymbol{S}}
\def\Cz{\boldsymbol{C}}

\providecommand\X[1]{\boldsymbol{X_{\hspace{-1pt}#1}}}
\providecommand\R[1]{\boldsymbol{R_{\hspace{-0pt}#1}}}

\providecommand\phib{\boldsymbol{\emptyset}}
\author{Corentin Briat\thanks{email:\textit{corentin.briat@bsse.ethz.ch}}, Christoph Zechner, Mustafa Khammash\thanks{email:\textit{mustafa.khammash@bsse.ethz.ch}}\\Department of Biosystems Science and Engineering, ETH Z\"{u}rich, Basel}

\title{Design of a synthetic integral feedback circuit: dynamic analysis and DNA implementation}
\date{}

\begin{document}

\maketitle

\begin{abstract} 
The design and implementation of regulation motifs ensuring robust perfect adaptation are challenging problems in synthetic biology. Indeed, the design of high-yield robust metabolic pathways producing, for instance, drug precursors and biofuels, could be easily imagined to rely on such a control strategy in order to optimize production levels and reduce production costs, despite the presence of environmental disturbance and model uncertainty.  We propose here a motif that ensures tracking and robust perfect adaptation for the controlled reaction network through integral feedback. Its metabolic load on the host is fully tunable and can be made arbitrarily close to the constitutive limit, the universal minimal metabolic load of all possible controllers. A DNA implementation of the controller network is finally provided. Computer simulations using realistic parameters demonstrate the good agreement between the DNA implementation and the ideal controller dynamics.

\end{abstract}


\section*{Introduction}

Adaptation to external disturbances is an essential requirement of many living cells \cite{Banci:13,Berg:15}. At the cellular level, adaptation is often achieved through regulatory strategies of various degrees of complexity. One common strategy is based on the well-known incoherent feedforward loop \cite{Alon:07,Ma:09}. Yet another is based on a negative feedback loop with a buffering node \cite{Ma:09}. Notably, the latter strategy was shown to employ integral feedback, a fundamental strategy extensively considered in control engineering \cite{Astrom:95,Albertos:10}. Integral feedback is particularly prevalent in control systems where a variable, the \emph{controlled variable}, needs to be regulated around a desired constant set-point, despite the presence of external disturbances. Integral feedback was shown to be the basis of a particularly strong type of adaptation, namely \emph{perfect adaptation}, which is the ability of a living system to internally compensate for the onset of a persistent environmental disturbance allowing it to return back to pre-disturbance levels after a transient \cite{Berg:15}. Perfect adaptation was observed in living organisms in a wide variety of contexts. Notable examples include blood calcium regulation in mammals \cite{ElSamad:02}, \emph{E. coli} chemotaxis \cite{Yi:00}, neuronal control of the prefrontal cortex  \cite{Miller:06} and the regulation of tryptophan in \emph{E. coli} \cite{Venkatesh:04} (see also the recent survey \cite{Somvanshi:15}). A novel integral control motif, referred to as \emph{antithetic integral controller}, that is particularly well-suited to noisy environments was recently proposed in \cite{Briat:15e}, where it was posited that this regulation motif is analogous to regulation mechanisms based on $\sigma$ and anti-$\sigma$ factors in bacteria \cite{Storz:11}.

In addition to the problem of elucidating the fundamental role of integral feedback in endogenous strategies for perfect adaptation, proposing novel synthetic motifs implementing an integral action may have an impact on the field of synthetic biology similar to that it had on the field of control engineering. An interesting potential application of integral feedback is the optimization of engineered metabolic networks; see e.g. \cite{He:13,Oyarzun:13,Berkhout:13}. Indeed, such a control strategy would allow for the efficient and guaranteed production of a metabolite of interest, such as a biofuel \cite{Savage:08,Fortman:08} or a drug precursor \cite{Ro:06,Paddon:14}, even in the presence of environmental disturbances and a poorly characterized metabolic pathway. However, theoretical design methods for integral feedback strategies have been seldom reported in the context of synthetic biology, where more elementary feedback control loops are often considered. For instance, the control of a microbial production process has been theoretically analyzed in \cite{Dunlop:10} where various non-integral control strategies were analyzed and compared.  Notably, it was shown that in spite of the fact that some of these strategies could improve the yield of the process through an appropriate tuning of the parameters, the controlled process did not exhibit any perfect adaptation properties. This implied a lack of robustness, a feature needed to compensate for the inherent variability of controlled biological systems. Approximate integral controllers obtained using Hill-kinetics have also been discussed in \cite{Ang:10,Ang:13}. However, the motifs considered in these references require fine parameter tuning in order to obtain a good performance for the controlled network, which is a tricky and time consuming procedure often based on trial-and-error. Another difficulty also lies in the large number of parameters involved, making the dependency of the properties of the controller an intricate function of these parameters, thereby adding a layer of complexity to the already fastidious parameter tuning problem. This contrasts with the antithetic integral controller of \cite{Briat:15e} for which the integral action is a \emph{structural property} of the motif,  meaning that this feature is independent of the values of the parameters, circumventing then the parameter tuning problem discussed above. On the experimental level, implementing synthetic integral feedback circuits remains largely an open problem,  in contrast to proportional-like control strategies, which have been implemented in several systems \cite{Stapleton:12,Auslander:14}.

The present paper is concerned with the analysis of a simple two-parameter integral feedback strategy for the control of biochemical networks. The considered integral feedback controller differs from the usual control engineering one (see Fig.~\ref{fig:topologies}) as we require it to be implementable in terms of chemical reactions and to yield a control signal, the \emph{control input} $u$, that is nonnegative. The latter requirement is, in general, not needed as control inputs are often allowed to take negative values. The integral action of this motif is shown to be \emph{structural}, as for the antithetic integral controller described in \cite{Briat:15e}. It is emphasized that the proposed integral controller can be used on a wide class of well-behaved networks, provided that one of the parameter of the controller, the stability coefficient $\alpha$ (see Fig.~\ref{fig:topologies}\textbf{B}), verifies a certain mild condition. When this condition is met, the trajectories of the controlled network converge towards the unique positive equilibrium point and the controlled network exhibits the desired robust perfect adaptation property for the controlled molecular species. A formula for the metabolic load onto the cell after implementation of the controller reactions is then obtained. We notably establish the existence of a lower bound for this metabolic load, referred here to as the \emph{constitutive limit}, which is shown to be equal to the metabolic cost of the actuation reaction. Note that this reaction is necessarily present, as it is the only reaction that would be involved if a naive constitutive control strategy were considered. Note that this control strategy is the simplest that can be used to control a network. In this regard, the difference between the actual metabolic cost and its lower bound can be interpreted as the \emph{cost of adaptation} for which it is proved that one can make it theoretically as close as desired from the constitutive limit through an appropriate tuning of the second controller parameter, the gain $k$ (see Fig.~\ref{fig:topologies}\textbf{B}), without deteriorating the tracking and perfect adaptation properties of the closed-loop network. A DNA implementation \cite{Soloveichik:10,Chen:13,Chen:15} of the controller reactions is then proposed as an anticipation over future in-vitro experiments. It seems important to stress that even though DNA implementations have been theoretically obtained in \cite{Yordanov:14} for linear modules, including proportional integral controllers, the considered (nonlinear) integral controller has never been designed using nucleic acids so far. Moreover, the spirit behind the current design is dramatically different from the one in \cite{Yordanov:14} as the goal in this latter paper was to report the design of DNA reactions that implement elementary linear systems using the approach proposed in \cite{Oishi:10b}. The current design, however, aims at proposing a generic integral controller that can either be used to regulate certain molecular levels in DNA reaction networks, or be implemented inside a cell to perform in-vivo control of biological reaction networks (see also \cite{Briat:15e}). In the former case, such a controller can be used as a regulating element in DNA reactions in order to achieve certain robustness and adaptation properties whereas, in the latter case, the current DNA implementation should be understood as a proof a concept for the proposed control strategy. Simulations of the DNA reactions using realistic parameters demonstrate the perfect adaptation properties of the network.

\begin{figure}[H]
  \centering
  \includegraphics[width=0.95\textwidth]{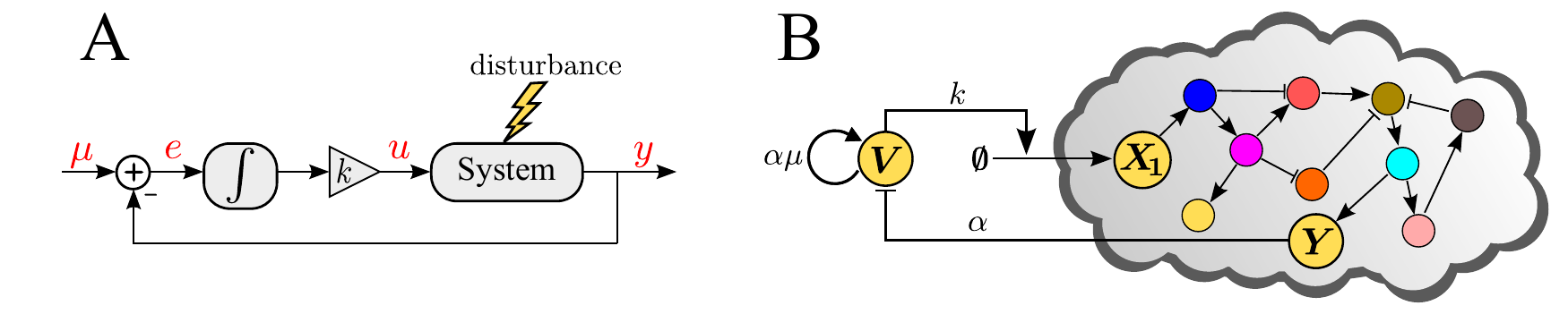}
  \caption{Comparison of the integral control topology of control engineering and the proposed one. \textbf{A.} The standard integral control topology used engineering integrates the error $e=\mu-y$ between the reference (the desired set-point) and the actual output) to yield the control input $u$ after a magnification of the result of the integration by a gain $k$. The control input is then used as an input to the system in order to steer the output $y$ to the desired set-point $\mu$ despite the presence of unmeasured disturbances acting on the system. \textbf{B.} The proposed integral controller indirectly computes the error between the reference $\mu$ and the concentration of the controlled species $\Yz$ through  the interplay between the autocatalytic reaction implementing the reference and the active degradation reaction of $\Vz$ which has a rate depending on the concentration of the controlled species $\Yz$. Finally, the actuated species $\X{1}$ is catalytically produced by the controller species $\Vz$ in order to steer the concentration of the controlled species to the desired steady-state value, despite the presence of stress acting on the controlled network.}\label{fig:topologies}
\end{figure}

\section*{Results and Discussion}

\textbf{The reaction network.} We consider here a reaction network $(\Xz ,\Rz)$ consisting of a family of $d$ distinct molecular species $\X{1},\dots,\X{d}$ which interact through $K$ reaction channels, denoted by $\R{1},\ldots,\R{K}$ (see \cite{Feinberg:72,Horn:72,Goutsias:13}). In the deterministic setting, the state of reaction networks is the vector of concentrations $x\in\mathbb{R}_{\ge0}^d$ of the molecular species $\Xz=(\X{1},\dots,\X{d})$. The evolution of these concentrations is governed by a system of differential equations called the \emph{reaction rate equations}, which can be written as:
\begin{equation}\label{eq:RNmodel1}
\begin{array}{rcl}
    \dot{x}(t) &=& f(x(t)),\ x(0) =x_0\\
    y(t)&=&e_\ell^Tx(t)
\end{array}
\end{equation}
where $e_\ell$ is the $d$-dimensional vector having the $\ell$-th entry equal to 1 ($\ell\in\{1,\ldots,d\}$) and the others to 0. The species $\Yz$ (with concentration $y(t)$ at time $t$) is called the \emph{controlled species} and is chosen, in this case, to coincide with $\X{\ell}$.  The function $f(x)$ can be directly obtained from the set of reactions and must satisfy, for all $i=1,\ldots,d$, the condition that $f_i(x,u)\ge0$ for all $x,u\ge0$ with $x_i=0$. This condition ensures that the state only takes nonnegative values over time and is automatically satisfied when considering reaction networks with mass-action kinetics, Hill-type kinetics, etc. Note that when the reactions are mass-action and are, at most, first-order reactions, then $f(x)$ can be written as $f(x)=Ax+b$ for some suitable matrix $A$ and vector $b$.

\textbf{The control problem and the controller.} \blue{We formulate here the control problem we would like to address and provide a controller that solves it. In order to be precise in our formulation, we use a control theoretic language that is devoted to stating and solving such problems. The different considered concepts, which may be unknown to the readers, will be systematically explained. We also refer the readers to the Box~1 in \cite{Briat:15e} for some additional details on control theory for systems and synthetic biology. The control problem is, therefore, the following: build a reaction network controller (i.e. a set of additional reactions and additional species) that interacts with the reaction network $(\Xz ,\Rz)$ by measuring the molecular concentration of the controlled species $\Yz$ and appropriately acting back on the network $(\Xz ,\Rz)$ in a way that}
\begin{enumerate}
  \item the positive equilibrium point of the closed-loop reaction network (consisting of the interconnection between the network $(\Xz ,\Rz)$ and the controller network) is asymptotically stable (i.e. the trajectories of the closed-loop network locally converge to the positive equilibrium point),
  \item the concentration $y(t)$ of the controlled species $\Yz\equiv\X{\ell}$ tracks a given constant reference $\mu>0$ \blue{(i.e. the concentration of the controlled species $y(t)$ tends to the desired set-point value $\mu$ as the time $t$ goes to infinity)},
  \item the closed-loop network exhibit \emph{perfect adaptation} for the controlled species $\Yz\equiv\X{\ell}$ \blue{(the concentration of the controlled species $y(t)$ always returns to the desired value $\mu$, even in presence of environmental disturbances such as changes in the values of the rate parameters of the network)}, and
  \item the controller has a low metabolic load on the host \blue{(i.e. the total power consumption -- in a way that will be defined later -- of the controller reactions is low)}.
\end{enumerate}

To solve the above control problem, we propose the following controller reaction network:
\begin{equation}\label{eq:controller}
  \underbrace{\Vz \rarrow{\alpha\mu}\Vz  + \Vz }_{\mbox{reference reaction}}\hspace{-0mm},\quad  \underbrace{\Vz  + \X{\ell}\rarrow{\alpha}\X{\ell}}_{\mbox{measurement reaction}}\ \textnormal{and}\ \ \ \underbrace{\Vz \rarrow{k}\Vz  + \X{1}}_{\mbox{actuation reaction}},
\end{equation}
where $\alpha,\mu$ and $k$ are three positive control parameters and where $\Vz$ is the so-called \emph{controller species}. The controller parameters $\mu$, $\alpha$ and $k$ are called the \emph{reference}, the \emph{stability coefficient} and the \emph{gain}, respectively. The first reaction is referred to as the \emph{reference reaction} as it encodes the value of the desired equilibrium value  $\mu$ for the controlled species. The second reaction is the \emph{measurement reaction} as it uses the current concentration of the controlled species $\Yz\equiv\X{\ell}$ to modulate the degradation rate of the controller species $\Vz $ (active degradation), thereby transferring some information on the current concentration level of the controlled species to the concentration level of the controller species. Finally, the last reaction is the \emph{actuation reaction} that catalytically produces the molecular species $\X{1}$, referred to as the \emph{actuated species}, using the current concentration level of the controller species $\Yz$. It is important to notice that this control motif, when connected to the network $(\Xz ,\Rz)$, implements a negative feedback loop, a characteristic that is prevalent in perfect adaptation motifs. \blue{Indeed, the controller species $\Vz$ activates $\X{1}$ which, after a cascade of reactions, will activate in turn the controlled species $\X{\ell}$ that will finally represses the controller species $\Vz$, thereby closing the loop negatively; see Fig.~\ref{fig:topologies}\textbf{B}.} Closely related control mechanisms have also been considered in \cite{Shoval:11} in the context of fold-change detection and in \cite{Buzi:15} in the context of stem cell differentiation even though, in the latter case, the model is considered at a cellular level rather than at a molecular level.

By interconnecting the network $(\Xz ,\Rz)$ and the controller network \eqref{eq:controller}, we obtain a closed-loop reaction network with reaction rate equations:
\begin{equation}\label{eq:closedloop}
\begin{array}{rcl}
\dot{x}(t)&=&f(x(t))+e_1kv(t)\\
\dot{v}(t)&=&\alpha v(t)(\mu-y(t))
\end{array}
\end{equation}
where $e_1$ is the $d$-dimensional vector having the first entry equal to 1 and the others to 0.

\textbf{Assumptions.} \blue{We will now make two assumptions for the network $(\Xz ,\Rz)$ in order to simplify the exposition of the results. Note, however, that these assumptions are principally considered to promote simplicity in the exposition of the results and are not necessary for the proposed controller to be applicable. This latter point will illustrated on dimerization network example (see Fig.~\ref{fig:dimer} and Section \ref{sec:SM:dimerization} of the Supporting Information).
\begin{description}
  \item[Assumption 1.] The reaction network is asymptotically stable and only contains first-order mass-action reactions.
  \item[Assumption 2.] The condition $e_\ell^TA^{-1}e_1\ne0$ holds.
\end{description}
The first assumption implies that $f(x)=Ax$ where $A$ is Hurwitz stable; i.e. all the eigenvalues of $A$ have negative real part whereas the second one is here to make sure  that the concentration level of $\Yz\equiv\X{\ell}$ can be actually changed by actuating $\X{1}$, i.e. there is sequence of reactions that propagate the changes in the concentration levels of $\X{1}$ to the concentration levels of $\X{\ell}$. This is equivalent, in the current case, to the concept of output controllability (see \cite{Briat:15e} for similar assumptions in the stochastic setting).}

\blue{The first assumption seems to be very restrictive since most of the biological networks are nonlinear as they involve higher-order mass-action kinetics or even saturating Hill-kinetics. However, most of the results developed in the linear case can be adapted to the nonlinear case by using the fact a smooth nonlinear system of ordinary differential equations can be locally represented around an equilibrium point by a system of linear differential equations: the so-called \emph{linearized system}. In this regard, linear analysis remains applicable when considering nonlinear systems. This will be demonstrated by the dimerization process example which we will consider later. Note also  that nonlinear versions of the above assumptions exist but are less explicit due to their nonlinear nature. Finally, it is important to mention that the first assumption can be slightly relaxed by allowing for the presence of zeroth-order mass-action reactions, leading then to a function $f$ of the form $f(x)=Ax+b$ for some nonnegative vector $b$. In this case, all the results obtained in the case $b=0$ can be extended to the case $b\ne0$ as shown in Section \ref{sec:SM:perfect_adaptation} of the Supporting Information. In the light of this discussion, the first assumption is hence not as restrictive as it seems, and considering it will allow us to establish and expose the results in a simple and clear way as generic closed-form formulas for the equilibrium points and other important properties characterizing the network properties will exist. This is the reason why, from now on, we will consider that the above assumptions are satisfied.}
%

\textbf{Local asymptotic stability.} Under the assumptions above, the model of closed-loop network \eqref{eq:closedloop} always admit two equilibrium points (see Section \ref{sec:SM:equilibrium_points} of the Supporting Information). The first one is the zero equilibrium point whereas the second one is a positive equilibrium and is given by
\begin{equation}
  (x^*,v^*)=\left(\dfrac{A^{-1}e_1\mu}{e_\ell^TA^{-1}e_1},\dfrac{-\mu}{e_\ell^TA^{-1}e_1k}\right),
\end{equation}
where $x_\ell^*=e_\ell^Tx^*=\mu$. As stated in Proposition \ref{prop:stab_zero_1} of the Supporting Information, the zero equilibrium point is always unstable (i.e. the trajectories of the model \eqref{eq:closedloop} never converge to the zero equilibrium point), meaning that the controller with always be active. The positive equilibrium point, on the other hand, is locally asymptotically stable (i.e. the trajectories of the model \eqref{eq:closedloop} converge towards the positive equilibrium point provided that they start sufficiently closely to it) whenever the stability coefficient $\alpha$ is smaller than a certain threshold that is inversely proportional to the reference $\mu$ (See Proposition \ref{prop:stab_positive_1} of the Supporting Information). Interestingly, the stability of the closed-loop network is independent of the value of the gain $k>0$, which can therefore be used for another purpose. The value of $\alpha$ can also be tuned so that the system has good convergence properties; e.g. a fast convergence rate towards the equilibrium or fast settling time (see Fig.~\ref{fig:gene:2b}). Note, moreover, that since the maximum admissible value for $\alpha$ is inversely proportional to $\mu$, then the range of admissible values for $\alpha$ gets smaller as the reference  increases (see Fig.~\ref{fig:gene:2b}\textbf{A}). This implies that increasing the equilibrium value for the controlled species has destabilizing effect on the closed-loop network. Finally, when the stability condition on $\alpha$ is not met, then the solutions of the system will be oscillatory, as it can be observed in Fig.~\ref{fig:gene:2b}.

\begin{figure}[H]
  \includegraphics[width=0.95\textwidth]{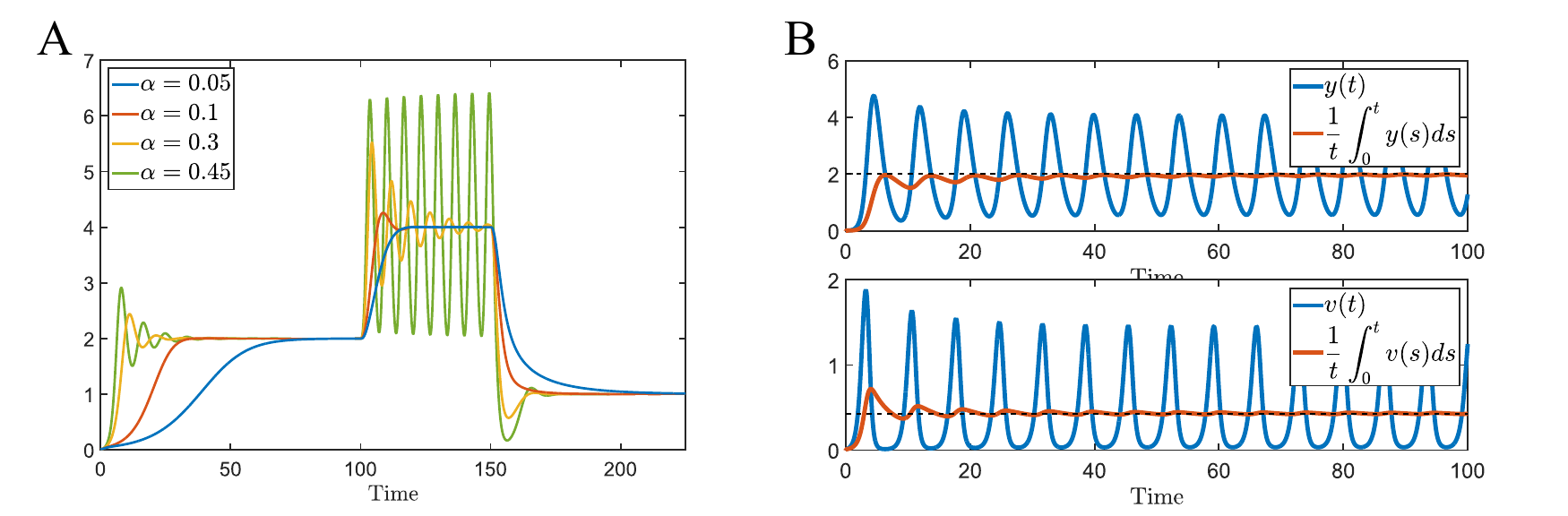}
  \caption{\small\textbf{A.} Evolution of the concentration of the mature protein  for various values for the reference $\mu$ and for the stability coefficient $\alpha$ illustrating that inappropriately chosen values for $\alpha$ may lead to oscillations. \textbf{B.} Even when the the concentration levels of the mature protein and the controller species are oscillatory, set-point regulation will still hold, but at the level of time-average of the concentration levels.}\label{fig:gene:2b}
\end{figure}

\textbf{Tracking and perfect adaptation properties.} The presence of the integral action ensures that the equilibrium value of the concentration of the controlled species is equal to the reference $\mu$, emphasizing then the property of constant reference tracking. By combining this statement with the (local) asymptotic stability of the positive equilibrium point, we can state that trajectories starting close enough from the equilibrium point will necessarily converge to it and that the concentration of the controlled species will, therefore, necessarily converge to the reference $\mu$. (see Fig.~\ref{fig:gene:1}\textbf{B}). The proposed controller, moreover, ensures that the equilibrium value for the concentration of the controlled species remains the same, i.e. equal to $\mu$, even in the presence of external disturbances (see Fig.~\ref{fig:gene:1}\textbf{C}) and of variations in the values of the network parameters (see Fig.~\ref{fig:gene:1}\textbf{D}), thereby demonstrating the perfect adaptation property for the closed-loop network (see Section \ref{sec:SM:perfect_adaptation} of the Supplementary Information). Finally, when the stability coefficient $\alpha$ does not meet the requirements of local asymptotic stability, the trajectories of the closed-loop network \eqref{eq:closedloop} will be oscillatory and, therefore, tracking and perfect adaptation will not hold. However, it is shown in Proposition \ref{prop:time_averages} of the Supplementary information that these properties will nevertheless hold at the level of the time-average of the trajectories. This is illustrated in Fig.~\ref{fig:gene:2b}\textbf{B}.

\begin{figure}[H]
  \centering
  \vspace{-1.5cm}
  \includegraphics[width=0.95\textwidth]{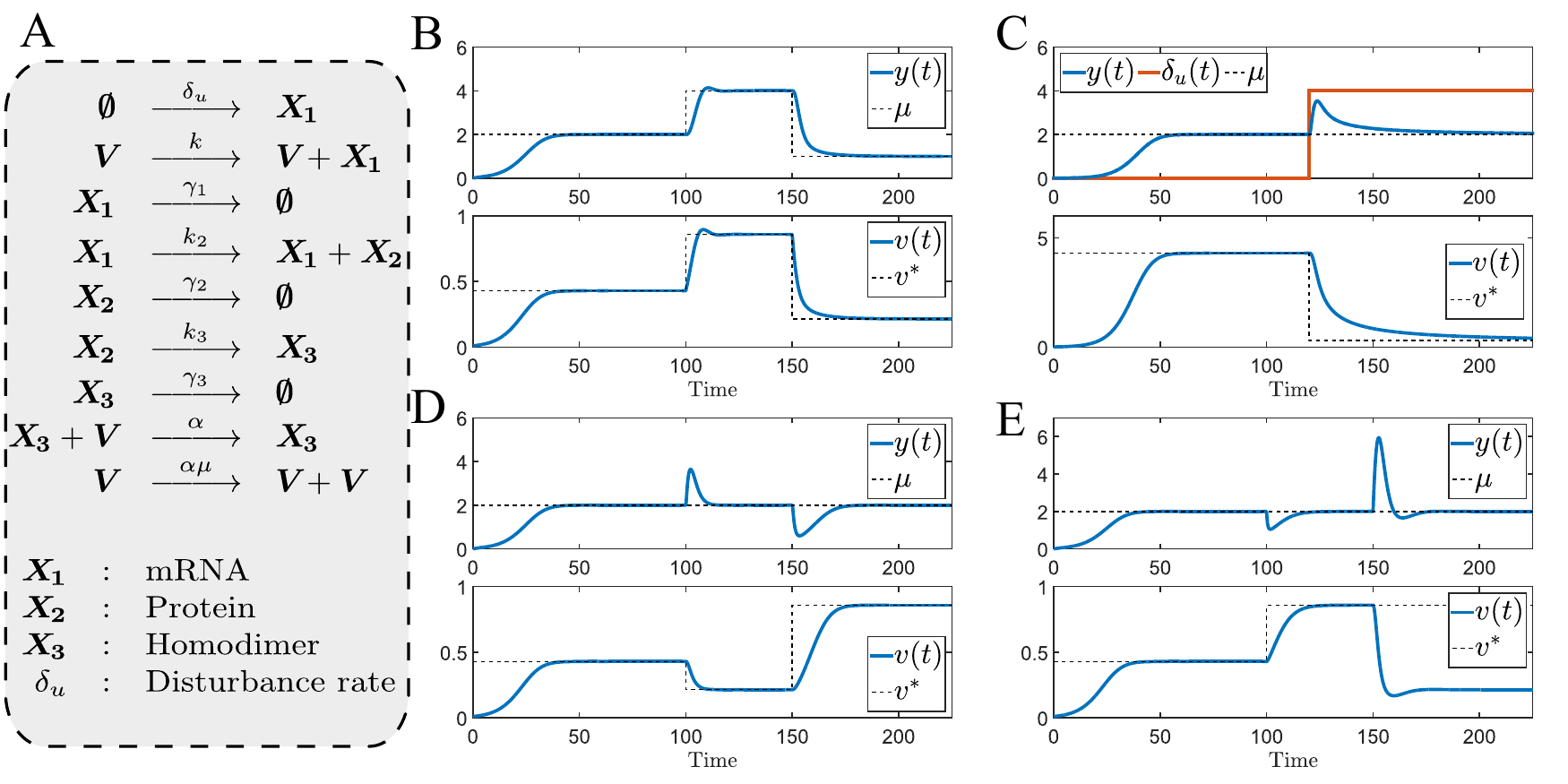}
  \caption{\small\textbf{A.} Reaction network describing the disturbed gene expression network with protein maturation and the various molecular species.  \textbf{B.} The integral controller ensures set-point regulation for the concentration of the mature protein (top) through a suitable and automatic modulation of the concentration of the controller species (bottom). \textbf{C.} The controller is able to detect the presence of the basal rate $\delta_u$  without directly measuring it and adjusts the controller species concentration level (bottom) to steer back the concentration of the mature protein to the desired set-point (top). \textbf{D.} Changes in the ribosomal activity induce variations for the translation rate $k_2$ which are detected by the controller (bottom) and immediately corrected to drive back the concentration of the mature protein to the desired set-point (top). \textbf{E.} The presence of downstream networks or changes in the enzyme activity affect the value of the degradation rate $\gamma_3$, which is immediately taken into account by the controller (top) to correct the trajectory of the mature protein concentration level.}\label{fig:gene:1}
\end{figure}

\textbf{Metabolic load and cost of adaptation.} It is interesting to characterize the metabolic load of the controller reactions on the host. Indeed, reactions that are too energy consuming will be detrimental, or even fatal, to the host. It is hence of general interest to develop biochemical controllers implementing reactions having low metabolic cost. It is proved (see Section \ref{sec:SM:metabolic_cost} of the Supplementary Information) that the power consumption at equilibrium $P^*$ is given by the expression
\begin{equation}
  P^*=\underbrace{\dfrac{\alpha\mu^2(\kappa_r+\kappa_m)}{kg}}_{\mbox{Cost of adaptation}}+\underbrace{\dfrac{\mu\kappa_a}{g}}_{\mbox{Constitutive limit}}
\end{equation}
where $g=-e_\ell^TA^{-1}e_1$ and $\kappa_r,\kappa_m$ and $\kappa_a$ are the elementary metabolic costs associated with the reference reaction, the measurement reaction and the actuation reaction, respectively. Not surprisingly, a large reference $\mu$ will be associated with a high metabolic load. This can, however, be compensated by increasing the controller gain $k$ without destabilizing the closed-loop network since, as mentioned in the section on closed-loop stability, the stability of the closed-loop network is independent of the value of the gain $k$. In the same way, when the system has high-gain $g\gg1$ (i.e. when small variations in the input are largely amplified to the output), the controller will be naturally less penalizing to the cell because the necessary concentration level for the controller species to reach will be low. However, this characteristic is tied to the network and cannot be a priori modified, unless a synthetic network is considered and designed to have a high gain. Finally, it is also important to emphasize the existence of the lower bound $\mu\kappa_a/g$ that may appear to impose a fundamental limitation to our control strategy. This is not the case, however, since this lower-bound is the smallest that we can get for the following reason: it corresponds to the cost of the actuation reaction that is necessarily present regardless of the type of controller network that is considered. This notably includes the constitutive controller that would constitutively produce $\X{1}$ at the constant rate $\mu/g$ (so that we indeed have that $x_\ell^*=\mu$ at equilibrium). This motivates the designation of this lower bound as the \emph{constitutive limit}. In this regard, we can argue that the constitutive controller structure (without feedback) is the least cumbersome, but no adaptation properties are ensured by it. The \emph{cost of adaptation}, which can be defined as the remaining part in the overall metabolic cost, can be theoretically made as close as desired from the constitutive limit $\mu\kappa_a/g$ by appropriately choosing $k$ (i.e. picking it large).

\textbf{Tuning the controller parameters.} The discussion above suggests that the parameters can be tuned in a hierarchical way. First of all, the range of values for the reference $\mu$ should be defined, and only then a suitable stability coefficient can be selected such that the closed-loop network has some good performance (e.g. fast convergence properties). The gain $k$ can be finally chosen to reduce the overall metabolic load of the network or to improve some transient behavior for the closed-loop network.

\textbf{Practical implementation.} As the controller reactions \eqref{eq:controller} may be difficult to implement per se, we propose the following alternative reaction network
\begin{equation}
    \Vz \rarrow{h_\theta(v)}\Vz  + \Vz, \ \Vz  + \X{\ell}\rarrow{\alpha}\X{\ell}\ \textnormal{and}\ \Vz \rarrow{k}\Vz  + \X{1}
\end{equation}
where $h_\theta(v)=\dfrac{\alpha\theta\mu}{\theta+v}$ and $\theta$ is chosen such that $\theta\gg1$. It can be shown that, in the limit $\theta\to\infty$, the above reaction network tends to the proposed one since $h_\theta(v)\to\alpha\mu v$ as $\theta\to\infty$ (while $v$ remains bounded). In the case of the controlled birth-death process, we have that the equilibrium point of the closed-loop network is given by
\begin{equation}
  (x^*_\rho,v_\rho^*)=\left(\dfrac{\mu\rho}{2}\left(-1+\sqrt{1+\dfrac{4}{\rho}}\right),\dfrac{\rho\mu\gamma}{2k}\left(-1+\sqrt{1+\dfrac{4}{\rho}}\right)\right)
\end{equation}
where $\rho=\dfrac{k\theta}{\mu\gamma}$. When $\rho\gg1$, we obtain that
\begin{equation}
  (x^*_\rho,v_\rho^*)\simeq\left(\mu,\dfrac{\mu\gamma}{k}\right),
\end{equation}
where we can recognize in the right-hand side the equilibrium point associated with the use of the integral controller \eqref{eq:controller}.

\textbf{Stochastic considerations.} The proposed controller can be proved to be unadapted to the stochastic setting because of the presence of an absorbing state for the controllser species. Indeed, the reaction $\Vz +\Yz \rarrow{\alpha}\Yz $ may completely deplete the system from the controller species $\Vz $, resulting in a failure of the controller. The biochemical noise is hence quite penalizing to this controller, unlike for the antithetic integral controller \cite{Briat:15e} that works both in the deterministic and the stochastic settings. \blue{However, the proposed controller involves one molecular species and three reactions in place of two molecular species and four reactions for the antithetic integral controller. This controller is therefore less complex and may be easier to implement than the antithetic integral controller. Note also that an integral controller which is less complex than the proposed one does not seem to exist since it involves the minimum number of molecular species and the supposedly minimum amount of reactions: one for actuating the system, one for measuring from the system and one implementing the controller integral dynamics.}

\blue{\textbf{Control of a gene expression network with maturation.} We illustrate here the proposed method on the gene expression network with maturation described in Fig.~\ref{fig:gene:1}\textbf{A}. Technical details can be found in Section \ref{sec:SM:gene_expression} of the Supplementary Information. For this network, the nominal values for the parameters are given by $\gamma_1=1.2337$, $k_{2} = 1.4513$, $\gamma_2=3.0155$, $k_{3} = 2.3679$, $\gamma_3=1.1114$, $\delta_u=0$, $\alpha=0.081$ and $k=10$. The rate $\delta_u$ is the basal transcription rate considered here as an unknown constant disturbance. Through an appropriate tuning of the parameter, we can see in Fig.~\ref{fig:gene:1}\textbf{B} that the concentration of the matured protein tracks the desired set-point $\mu$ over time through a suitable and automatic correction of the concentration of the controller species and as a result of the considered integral feedback strategy. The value of $\mu$ is initially set to 2. At $t=100$, this values increases to 5 and, finally, decreases to 1 at $t=150$.}

\blue{We now analyze the influence of the basal transcription rate $\delta_u$ in the case $k=1$. The initial value for $\delta_u$ is 0 and increases to 4 at $t=120$. We can see in Fig.~\ref{fig:gene:1}\textbf{C} that the controller is able to detect the increase in the concentration of the controlled species and to compensate it by reducing the concentration level of the controller species, therefore ensuring  perfect adaptation for the mature protein concentration level, which is again a consequence of integral feedback. The influence of variations of the translation rate $k_2$ on the mature protein and controller species concentration levels is depicted in Fig.~\ref{fig:gene:1}\textbf{D} where, initially set to 1.4513, the translation rate $k_2$ is increased to 2.9026 at $t=100$ and then decreased to 0.7257 at $t=150$. The variation of this parameter can be, for instance, interpreted as a change in the ribosomal activity within the cell. We can observe, once again, that the concentration level of the mature proteins perfectly adapts to variations of this parameter, a consequence of the ability of the controller to detect these variations and correct them by modulating the concentration levels of the controller species. Finally, the influence of variations of the mature protein degradation rate $\gamma_3$ on the mature protein and controller species concentration levels is depicted in Fig.~\ref{fig:gene:1}\textbf{E} where, initially set to 1.1114, the degradation rate $\gamma_3$ is increased to 2.2228 at $t=100$ and then decreased to 0.5557 at $t=150$. The variability of this rate constant can be a consequence of a loading effect arising from the presence of downstream network with variable activity that takes mature proteins as input and converts them into another molecular species.  We can see that the concentration level of the mature proteins perfectly adapts to the variations of this parameter, a consequence of the ability of the controller to detect these variations and to appropriately correct the levels of the controller species.}

\blue{To analyze the influence of the stability coefficient $\alpha$ on the dynamics of the closed-loop network we dynamically change its value in the simulation. We can see in Fig.~\ref{fig:gene:2b}\textbf{A} that when $\alpha$ increases, the output responds faster but also moves towards instability (oscillations). When $\mu$ increases from 2 to 4 at $t=100$, we can see the that asymptotic stability is lost for the trajectory corresponding to the stability coefficient equal to $\alpha=0.45$, indicating then that this value exceeds the maximal admissible value that preserves the stability of the closed-loop network.  This can be theoretically explained using the facts that, when $\mu=2$, the stability is preserved provided that $\alpha<0.84$ and, when $\mu=4$, this is only the case when $\alpha<0.42$. This illustrates the fact that the upper-bound on the reference $\mu$ should be considered before choosing the value of the stability coefficient, as mentioned in the section on parameter tuning. Finally, when $\mu=1$, the stability of the closed-loop network is ensured if $\alpha<1.68$, which explains why all the trajectories are well-behaved (i.e. not oscillatory) after $t=150$. Finally, we illustrate in Fig.~\ref{fig:gene:2b}\textbf{B} the fact that even when the concentration levels of the mature protein and the controller species exhibit an oscillatory (i.e. not asymptotically stable) behavior, we can observe that the time-averages both converge to the desired equilibrium point, emphasizing that tracking is achieved at the level of the time-averages of the trajectories. Although not represented, the time-average of the concentration of the controlled species will also exhibit perfect adaptation properties.}

\blue{\textbf{Control of a dimerization network.} We demonstrate here that the proposed integral controller also works on nonlinear networks such as the dimerization process depicted in Fig.~\ref{fig:dimer}\textbf{A} where the species $\X{1}$ and $\X{2}$ represent the protein molecules and the corresponding homodimer, respectively. The technical results can be found in Section \ref{sec:SM:dimerization} of the Supplementary Information. The parameters of the network are set to $\gamma_1=1$, $k_{12} = 1$, $\gamma_2=2$, $k_{21} = 2$, $\alpha=0.2$ and $k=10$. The goal here is to act on the production rate of the protein $\X{1}$ in order to regulate the concentration of the homodimer $\X{2}$. We can see in  Fig.~\ref{fig:dimer}\textbf{B} that when the controller parameters are chosen appropriately, the concentration of the homodimer is able to track the desired set-point through an appropriate adaptation of the production of the protein by the controller network. There, the value for $\mu$ initially set to 2, is increased to 5 at $t=50$s and decreased to 1 at $t=100s$.  In a similar way, we can see in Fig.~\ref{fig:dimer}\textbf{C} that the controller is able to drive back the concentration of the homodimer to the desired set-point whenever parameters in the network vary. In this simulation scenario, the value of $\gamma_1$, initially set to 1 is increased to 3 at  $t=50$ and decreased to 1/2 at $t=100$. This demonstrates the ability of the proposed to controller to ensure perfect adaptation for the controlled nonlinear network.}

\begin{figure}[H]
  \centering
  \includegraphics[width=0.95\textwidth]{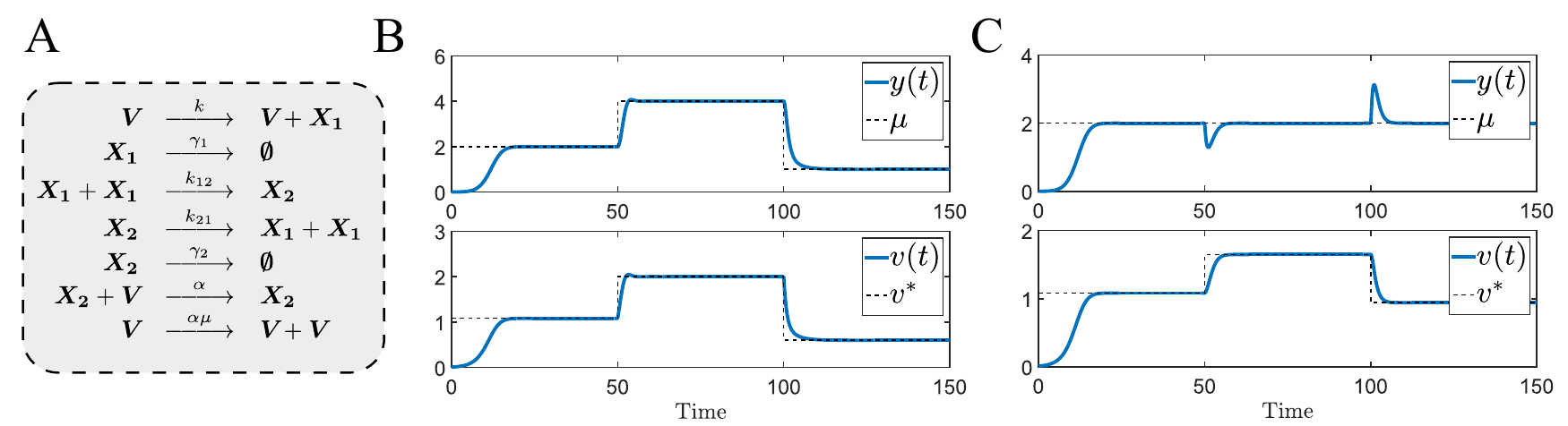}
  \caption{\small\textbf{A.} Reaction network describing a dimerization network controlled with the proposed integral controller. The species $\X{1}$ and $\X{2}$ represent the protein molecules and the corresponding homodimer, respectively. The controller monitors the concentration levels of the homodimer and acts back on the production rate of the protein. \textbf{B.}  Evolution of the concentration of the homodimer $\X{2}$ (top) and the controller species $\Vz$ (bottom) for various values for the reference $\mu$. The value for $\mu$ is increased at $t=50$ and decreased at $t=100$. We can see that the concentration levels of the controlled species tracks the desired set-point through the adjustment of the concentration levels of the controller species as a response of variations in the concentration levels of the controlled  species. \textbf{C.} Evolution of the concentration of the homodimer $\X{2}$ (top) and the controller species $\Vz$ (bottom) for various values for the parameter $\gamma_1$. The variation of this rate constant can be interpreted as the presence of unmodeled network that can, for example, converts the proteins into another molecular species at a varying rate. The value of $\gamma_1$ is increased at $t=50$ and decreased  at $t=100$. We can see that the concentration level of the homodimer molecules perfectly adapts to the variations of this parameter, a consequence of the ability of the controller to detect these variations and to react accordingly by modulating the concentration levels of the controller species.}\label{fig:dimer}
\end{figure}

%
%

\textbf{DNA implementation.} \blue{In the following, we show how the proposed control circuit can be realized biochemically using DNA strand displacement (DSD) cascades \cite{Phillips:09,Soloveichik:10,Zhang:11,Qian:11,Chen:13,Chen:15}. Such cascades allow us to incorporate complex circuit dynamics through a network of hybridization reactions involving single- and multi-stranded DNA molecules (see \cite{Chen:15} for an overview). We consider here the DSD design strategy from \cite{Soloveichik:10} allowing us to synthesize arbitrary mass-action kinetics from DNA. The main idea is to map each formal reaction of the controller to multiple bimolecular reactions in the corresponding DNA circuit. More specifically, one uses a multi-stranded DNA complex (referred to as gate) to first recruit all reactants of a formal reaction. Once all reactants are in place, the gate displaces a messenger strand, which binds to a second, which in turn releases all products of the formal controller reaction.}

\blue{The system and controller species $\Xz$ and $\Vz$ are represented by single-stranded DNA molecules consisting of two short binding domains of around $5nt$ (referred to as toehold domains) and a single specificity domain of around $20nt$. 
The additional gate complexes are supplied in excess, such that their concentration (initially set to $\Omega$) is substantially higher than the concentrations of the network and the controller species (i.e. $x(t)$ and $y(t)$) and will hence remain effectively unchanged over the duration of an experiment. It can then be shown analytically that the higher-dimensional DNA kinetics (see Fig.\ref{fig:BD:DNA}) simplify to the desired rate laws of the formal reactions. For instance, the autocatalytic reference reaction
\begin{equation*}
	\Vz \rarrow{k} \Vz + \Vz
\end{equation*}
is realized biochemically using two bimolecular DNA reactions
\begin{equation*}
\begin{split}
	\Vz + \Cz &\rarrow{\lambda_{1}} \Uz + \text{waste}\\
	\Uz + \Sz &\rarrow{\tilde{\lambda}} \Vz + \Vz + \text{waste},\\
\end{split}
\end{equation*}
with $\lambda_{1}$ and $\tilde{\lambda}$ as physical reaction rate constants and with ``waste'' serving as a proxy symbol to denote non-reactive byproducts that are created during strand displacement (see Figure S.1 in the Appendix for a specific example). Now, if  $c(0)=s(0)=\Omega>>v(t)$, we can assume that $c(t)\approx \Omega$ and $s(t)\approx \Omega$ over the duration of an experiment. 
Using quasi-steady state analysis it can then be shown that the rate of the aggregated reaction rate simplifies to $\lambda_{1} \Omega^{-1}v(t)$, which is consistent with the desired rate law of the autocatalytic reference reaction. A similar procedure can be followed to derive the kinetics of the measurement and actuation reactions (see \cite{Soloveichik:10} for a full kinetic analysis of uni- and bimolecular reactions).}

\begin{figure}[H]
  \centering
  \includegraphics[width=0.95\textwidth]{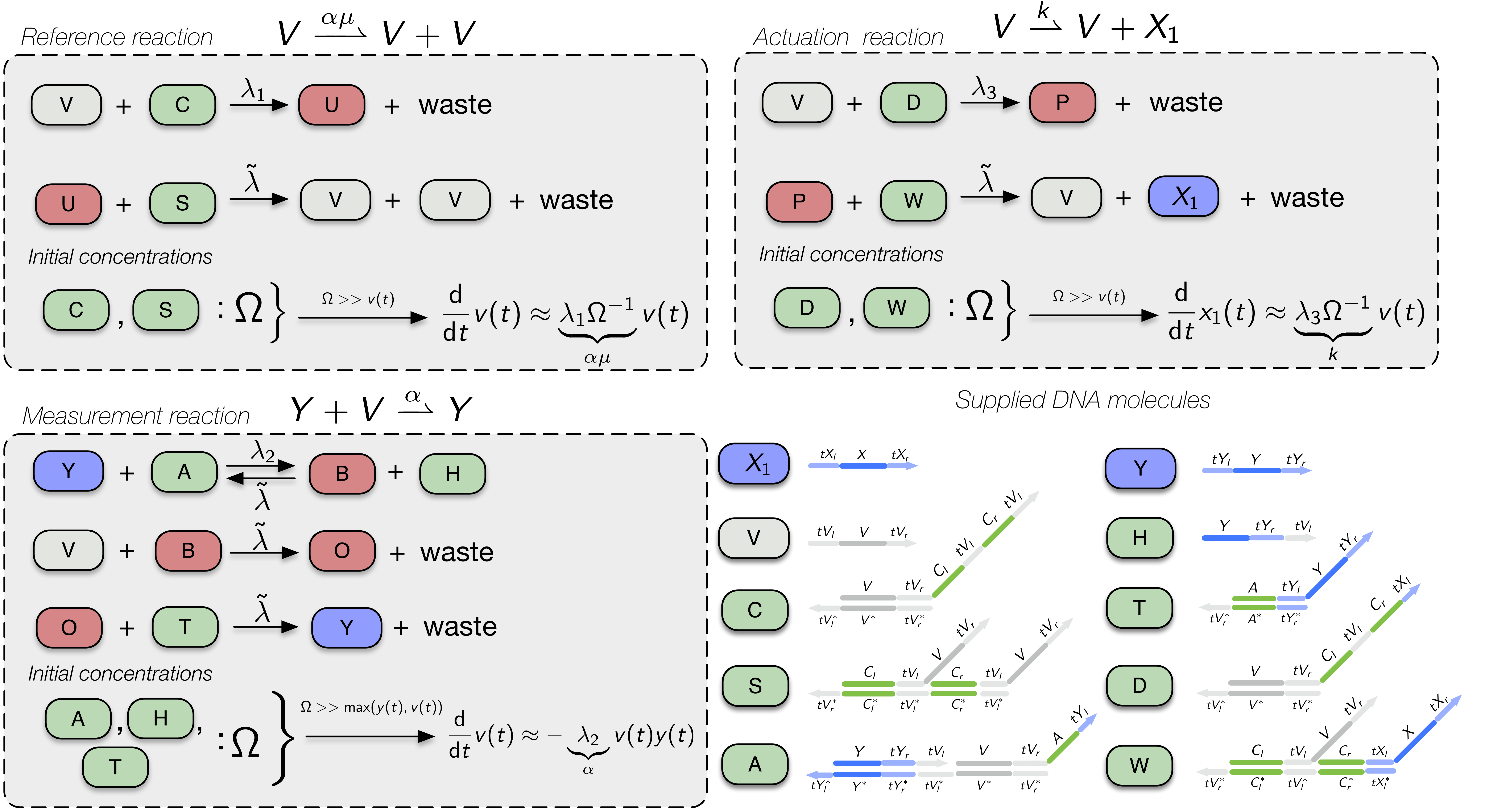}
  \caption{\small Integral controller implemented as a DNA strand displacement cascade. Signal strands (e.g. $\X{1},\ldots,\X{d}$, $\Vz$) consist of two distinct toehold domains (indicated by symbols starting with $t$) and one longer specificity domain (upper case symbols). Single-stranded signal molecules specifically bind to the corresponding multi-stranded complexes through Watson-Crick base pairing. Once bound, there is a certain chance that the originally bound top strand is displaced (i.e., branch migration) and replaced by the matching incoming signal strand. 
  Specific controller parameters can be achieved by modifying DNA sequences and/or initial concentrations of multi-stranded gate complexes. For simplicity, we assume all rate constants to be $\tilde{\lambda}=0.01 nM s^{-1}$ except those that need to be tuned to achieve the desired parameters (i.e., $\lambda_{1}-\lambda_{3}$). The legend shows the molecules that have to be initially added to the test tube. Intermediate species that are created as the system evolves are not shown in the legend.}\label{fig:BD:DNA}
\end{figure}

\blue{We performed computer simulations of the DNA-based controller for a system that is described by a one-dimensional death process $\Xz\rarrow{\gamma}\phib$ where the controlled species $\Yz$ coincides with $\Xz$ and hence $y(t)=x(t)$. By interconnecting the network and the controller together, we obtain the closed-loop network given, in terms the formal reactions, by
\begin{equation*}
\begin{array}{rcl}
  \Xz&\rarrow{\gamma}&\phib,\\
    \Vz&\rarrow{k}&\Vz+\Xz,\\
  \Xz+\Vz&\rarrow{\alpha}&\Xz,\\
  \Vz&\rarrow{\alpha\mu}&\Vz+\Vz.
\end{array}
\end{equation*}}
%

\blue{Using realistic binding rates from the literature \cite{Soloveichik:10}, the circuit was dimensioned such that it can be easily implemented and validated experimentally. As in our previous simulations, we checked the controller performance using various case studies related to tracking and adaptation (see Fig.~\ref{fig:BD:tracking}\textbf{a-c}). We found that the DNA circuit agrees very well with the ideal controller kinetics. Since multi-stranded DNA molecules are consumed during the hybridization steps, the circuit will ultimately loose its designed characteristics. This effect is particularly strong for small $\Omega$ as shown in Fig.~\ref{fig:BD:tracking}\textbf{c}.}

\begin{figure}
    \centering
 \includegraphics[width=0.95\textwidth]{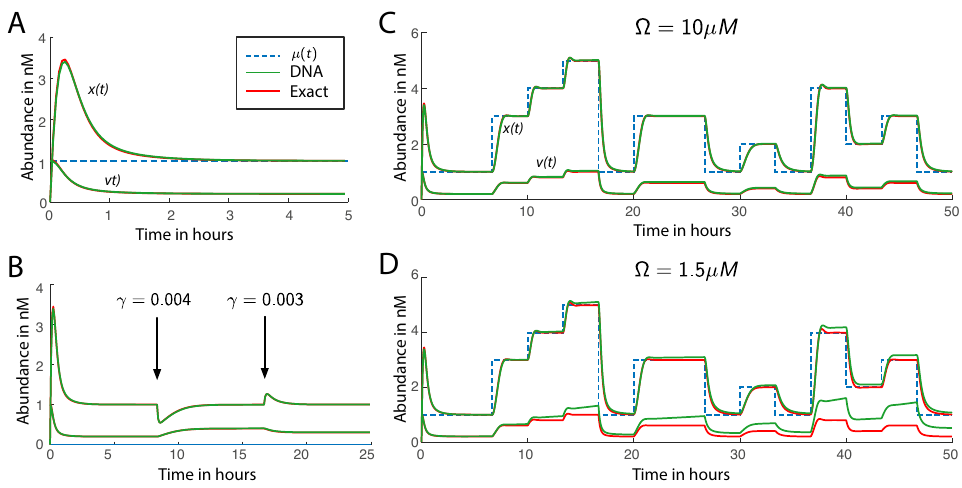}
    \caption{\small Simulation results of a DNA-based integral controller controlling a one-dimensional process $x(t)$, that degrades with rate $\gamma$, through its birth-rate (note that we have here $\Xz=\X{1}=\Yz$ as the controlled process consists of a single molecular species). Unless otherwise stated, the parameters $V(0) = 1nM$, $X(0) = 0nM$, $k=0.01s^{-1}$, $\gamma = 0.002 s^{-1}$, $\alpha = 3e-4 nM s^{-1}$, $\mu=1nM$, $\Omega=10\mu M$, $\tilde{\lambda}=0.01 nM s^{-1}$ are considered here. All plots show both $X(t)$ and the controller $V(t)$ for the exact and the DNA-based implementation. As long as the initial gate supply $\Omega$ is large enough, the DNA-based circuit resembles the exact dynamics at a remarkable precision. \textbf{A.} Transient behavior for a fixed $\mu=1nM$. \textbf{B.} Adaptation. The dynamics of $X(t)$ are perturbed at discrete time points by changing the value of $\gamma$ to $0.004 s^{-1}$ and $0.003 s^{-1}$, respectively. \textbf{C.} Long-term tracking behavior. The controlled death process is simulated for a randomly generated profile for the reference $\mu(t)$ and with $\Omega=10\mu M$. In this scenario, the kinetics of the DNA circuit agrees very well with the exact solution. \textbf{D.} Long-term tracking behavior. The controlled death process is simulated for a randomly generated profile for the reference $\mu(t)$ and with $\Omega=1.5\mu M$. In this scenario, the kinetics of the DNA circuit deviates quickly from the exact solution, illustrating then the necessity of supplying a high enough initial gate concentration $\Omega$.}\label{fig:BD:tracking}
\end{figure}

\section*{Acknowledgements}

The authors acknowledge funding support from the Swiss National Science Foundation grant 200021-157129. CZ is funded by a grant from the Swiss SystemsX.ch initiative, evaluated by the Swiss National Science Foundation.

\newpage

\renewcommand{\thesection}{S\arabic{section}}
\renewcommand{\thetheorem}{\thesection.\arabic{theorem}}
\renewcommand{\thefigure}{S\arabic{figure}}

\numberwithin{equation}{section}

\begin{center}
  {\LARGE Design of a synthetic integral feedback circuit: dynamic analysis and DNA implementation\\ --\\ Supplementary Material}\\\ \\
  {\large Corentin Briat, Christoph Zechner, Mustafa Khammash}
\end{center}

\vfill
\tableofcontents
\vfill

%

\newpage

\section{Preliminaries}

\subsection{Deterministic reaction networks}

Reaction networks \cite{Feinberg:72,Horn:72,Goutsias:13} are very powerful modeling tools that are able to describe a wide variety of processes. A reaction network consists of a family of $d$ distinct molecular species $\X{1},\dots,\X{d}$ which interact through $K$ reaction channels, denoted by $\R{1},\ldots,\R{K}$. In the deterministic case, each reaction is described by a \emph{stoichiometric  vector} $\zeta_k = (\zeta_{k,1},\dots,\zeta_{k,d})$, and a propensity function $\lambda_k(x)$. The stoichiometric vector describes here the direction of change of the state under the influence of the reaction while the propensity function describes the strength of the corresponding reaction (the modulus of the vector).  The state of the process is denoted by $x\in\mathbb{R}^d$ and its evolution is governed by the Reaction Rate Equations (RRE) defined as
\begin{equation}\label{eq:RNmodel1}
\begin{array}{rcl}
  \dot{x}(t) &=&  \sum_{k=1}^K \lambda_k(x(t)) \zeta_k,\ x(0) =x_0\\
                    &=:&f(x),\ x(0) =x_0
\end{array}
\end{equation}
For biological models or, more generally, population models, the state should remain nonnegative and, therefore, the overall should satisfy the condition that, for all $i=1,\ldots,d$, we have that $f_i(x)\ge0$ for all $x\ge0$ with $x_i=0$. In this regard, the propensity functions should be accommodated with this constraint which is naturally satisfied for systems with mass-action kinetics. This class of systems is known as \emph{positive systems} and have been considered for the modeling of a wide variety of real world processes such as the modeling of populations \cite{Murray:02}, biological \cite{Briat:12c,Briat:13h}, communication networks \cite{Shorten:06,Briat:13f}, etc.\\

\noindent When the reactions  $\R{1},\ldots,\R{K}$ are, at most, unimolecular (i.e. zeroth- and first-order reactions), then the dynamics of the network is described by the following model
\begin{equation}\label{eq:RNmodel2}
\dot{x}(t) =  Ax+b,\ x(0) = x_0.
\end{equation}
where $A\in\mathbb{R}^{d\times d}$ is Metzler (i.e. the off-diagonal elements are nonnegative) and $b\in\mathbb{R}^d$. The fact that the matrix $A$ is Metzler comes from the fact that the system is a linear positive system \cite{Farina:00}.

%
%

\subsection{The control problem (unimolecular network case)}

The setup addressed in this paper is the following. Assume that we are given a unimolecular reaction network having the following state-space representation
\begin{equation}\label{eq:mainsystL}
  \begin{array}{lcl}
    \dot{x}(t)&=&Ax(t)+Bu(t)\\
    y(t)&=&Cx(t)\\
    x(0)&=&x_0
  \end{array}
\end{equation}
where $x\in\mathbb{R}_{\ge0}^n$ is the state of the system, $u\in\mathbb{R}$ is the control input and $y\in\mathbb{R}$ is the measured output that has to be controlled. Note that the system has a single input and a single output (a so-called SISO system). As explained above the matrix $A$ is Metzler since the system is positive. We moreover assume that the matrices $B$ and $C$ are nonnegative as well. In this case, the state $x(t)$ and the output  $y(t)$ will be nonnegative only if $u(t)$ is nonnegative as well. This is consistent with the fact that the control input will be taken as a rate constant in the nonnegative vector $b$ in \eqref{eq:RNmodel2};  see e.g. \cite{Briat:15e} for a similar setup in the stochastic setting.\\

\noindent We are now in position to state the control problem we will be interested in:
\begin{problem}\label{prob:1}
  Let $\mu>0$. Find a controller, that measures the output $y$ to act on the system \eqref{eq:mainsystL} by modulating the value of the control input $u$, such that
  \begin{enumerate}[(St1)]
   \item the control input $u$ is nonnegative at all times,\label{controlobj:1}
   \item the controller can be implemented in terms of a reaction network with mass-action kinetics, and\label{controlobj:2}
    \item the closed-loop network, consisting of the interconnection of the network described by \eqref{eq:mainsystL} and the controller, has the following properties\label{controlobj:3}
\begin{enumerate}[(St3a)]
    \item the output $y$ is regulated around the fixed set-point $\mu>0$; i.e. $y(t)\to y^*=\mu$ as $t\to\infty$;
    \item the (unique) equilibrium point $x^*$ associated with the value $y^*=\mu$ is (locally) asymptotically stable;
    \item the output $y$ perfect adapt in response to changes in the parameters (robustness) and the presence of constant disturbances acting on the input and on the state of the system (disturbance rejection).
\end{enumerate}
    \end{enumerate}
\end{problem}
The objectives in statement \emph{(St\ref{controlobj:3})} are very standard in control theory. Indeed, the closed-loop network will (in general) satisfy statement \emph{(St\ref{controlobj:3})} if the controller implements an integral action and the parameters of the controller are chosen such that the stability of the closed-loop system is ensured. However, the two first statements are less standard since it is usually not required that the control input be nonnegative for all time and that the controller be implementable in terms of a reaction network with mass-action kinetics. These constraints are specific to our control problem and will be enforced in Section \ref{sec:integral}.\\

\noindent We now give an illustrative example for our control problem:
\begin{example}
  Let us consider the gene expression network
\begin{equation}
  \phib\rarrow{k_1}\X{1},\   \X{1}\rarrow{\gamma_1}\phib,\ \X{1}\rarrow{k_2}\X{1} + \X{2}\ \textnormal{and}\ \X{2}\rarrow{\gamma_2}\phib
\end{equation}
where $\X{1}$ and $\X{2}$ represent mRNA and the associated protein molecules, respectively. The reaction rates $k_1,\gamma_1,k_2$ and $\gamma_2$ are positive real numbers. Assume now that we would like the protein concentration to be regulated around some steady-state value $\mu>0$. By then considering the rate $k_1$ to be our control input $u$, we get the model
\begin{equation}
  \begin{array}{rcl}
  \begin{bmatrix}
  \dot{x}_1(t)\\
\dot{x}_2(t)
\end{bmatrix}&=&\begin{bmatrix}
  -\gamma_1 & 0\\
k_2 & -\gamma_2
\end{bmatrix}\begin{bmatrix}
  x_1(t)\\
x_2(t)
\end{bmatrix}+\begin{bmatrix}
  1\\
0
\end{bmatrix}u(t)\\
y(t)&=&\begin{bmatrix}
  0 & 1
\end{bmatrix}x(t)
\end{array}
\end{equation}
where $x_1(t)$ and $x_2(t)$ are the concentrations of the mRNA and protein molecules at time $t$, respectively. We can then immediately see that this control problem falls in the framework considered in this section.
\end{example}

\begin{remark}
  It is important to stress that we have restricted our attention only to unimolecular networks only for the sake of developing a general theory. The problem \ref{prob:1}, however, does not make any assumption on the dynamics of the reaction network to be controlled. It will be also shown later that the proposed controller can also be applied to more general reaction networks, emphasizing then its generality.
\end{remark}

\section{Integral control and nonlinear positive integral control}\label{sec:integral}

We quickly review here the benefits of integral control and explain why there is a problem when such an integral control structure is considered for controlling a system restricted to have a nonnegative control input. We also emphasize the problem of the implementing the integral controller in terms of a reaction network with mass-action kinetics.

\subsection{Integral control of linear systems}

Integral control \cite{Astrom:95} is a cornerstone of control theory as it allows to ensure that the output of the controlled system will reach the desired state value provided that the closed-loop system is stable. On the top of that, constant disturbances on the input and on the state will be automatically rejected and sufficiently small plant perturbations will not affect the steady state value of the controlled output. \\

\noindent To illustrate this, let us consider a SISO (not necessarily positive) linear system $\Sigma$ described by the rational strictly-proper transfer function $G(s)$. For convenience, we assume that the system is asymptotically stable and that $G(0)\ne0$. Note that, in this case, there exist an integer $n$ and matrices $A\in\mathbb{R}^{n\times n}$, $B\in\mathbb{R}^{n\times 1}$ and $C\in\mathbb{R}^{1\times n}$ such that the system
\begin{equation}\label{eq:dkskdsqmldksqldm}
\begin{array}{rcl}
  \dot{x}&=&Ax+Bu\\
    y&=&Cx
\end{array}
\end{equation}
is a realization of the process $\Sigma$; i.e. we have that $G(s)=C(sI-A)^{-1}B$. In other words, the Laplace transform of the impulse response $h(t)=Ce^{At}B$ is equal to the transfer function $G(s)$. This emphasizes the correspondence between the transfer function considered in this section and the state-space representation considered in the previous section.\\

\noindent We now show that the integral controller
\begin{equation}\label{eq:intdsmdds}
\begin{array}{rcl}
     u(t)&=&kI(t)\\
    \dot{I}(t)&=&\mu-y(t)
\end{array}
\end{equation}
where $k$ is the gain of the controller to be designed, solves the objective problems in Problem \ref{prob:1}, Statement \emph{(St\ref{controlobj:3})}. Note that the transfer function of the map $(\mu-y)\mapsto u$ is given by $C(s):=k/s$ as this is an integrator.

By interconnecting the system $\Sigma$ having realization \eqref{eq:dkskdsqmldksqldm} with the controller \eqref{eq:intdsmdds} yields the control loop depicted in Fig.~\ref{fig:loopnophi} where two additive disturbances are also depicted. The input disturbance  $\delta_u$ acts on the input while the output disturbance $\delta_y$ acts on the output.

\begin{figure}[h]
  \centering
  \includegraphics[width=0.75\textwidth]{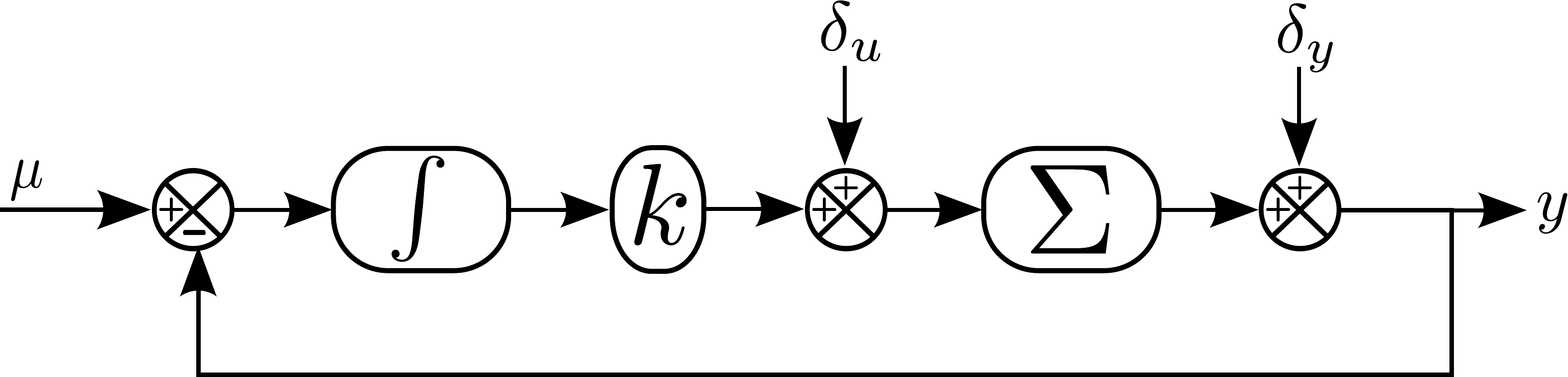}
  \caption{Integral feedback loop.}\label{fig:loopnophi}
\end{figure}

We propose to show now that the output will track the references and that constant disturbances will be rejected. The following expression for the output can be obtained
\begin{equation}
  \widehat{y}(s)=\dfrac{G(s)C(s)}{1+G(s)C(s)}\dfrac{\mu}{s}+\dfrac{G(s)}{1+G(s)C(s)}\widehat{\delta}_u(s)+\dfrac{1}{1+G(s)C(s)}\widehat{\delta}_y(s)
\end{equation}
where $\widehat{\cdot}$ denotes the Laplace transform of the corresponding signals. Assuming that the disturbances are constant, we then have that $\widehat{\delta}_u(s)=d_u/s$ and $\widehat{\delta}_y(s)=d_y/s$ for some $d_u,d_y\in\mathbb{R}$. We then get that
\begin{equation}
  \widehat{y}(s)=\dfrac{G(s)k}{s(s+G(s)k)}\mu+\dfrac{G(s)}{s+G(s)k}d_u+\dfrac{1}{s+G(s)k}d_y.
\end{equation}
Assuming then that $k$ is such that the closed-loop system is asymptotically stable, then we can apply the final value theorem which states that
\begin{equation}
  \lim_{t\to\infty}y(t)=\lim_{s\to0}s\widehat{y}(s)
\end{equation}
to get that
\begin{equation}
\begin{array}{lcl}
  \lim_{t\to\infty}y(t)&=&\underbrace{\lim_{s\to0}\dfrac{G(s)k}{s+G(s)k}\mu}_{\mbox{$\mu$}}+\underbrace{\lim_{s\to0}\dfrac{sG(s)}{s+G(s)k}d_u}_{\mbox{$0$}}+\underbrace{\lim_{s\to0}\dfrac{s}{s+G(s)k}d_y}_{\mbox{$0$}}\\
  &=&\mu
\end{array}
\end{equation}
where we have used the fact that $G(0)\ne0$ by assumption. We can then readily see that, as desired, the asymptotic value for the output $y$ is equal to $\mu$, and is independent of the constant disturbance levels $d_u$ and $d_y$. It is also interesting to see the behavior of the control input with respect to the disturbances. Applying standard tools, we get that
\begin{equation}
  \widehat{u}(s)=\dfrac{k}{s+kG(s)}\dfrac{\mu}{s}-\dfrac{kG(s)}{s+kG(s)}\dfrac{d_u}{s}-\dfrac{k}{s+kG(s)}\dfrac{d_y}{s}.
\end{equation}
Using again the final value theorem, we get that
\begin{equation}
  \lim_{t\to\infty}u(t)=\dfrac{\mu}{G(0)}-d_u-\dfrac{d_y}{G(0)}
\end{equation}
where we can see that the integral controller finds the correct equilibrium control input value so that regulation of the output around the desired set-point value is achieved. Interestingly, we can see that the controller is able to estimate the static-gain $G(0)$ of the system and the values of the amplitudes of the constant disturbances $d_u$ and $d_y$ without any prior knowledge about them. All what needs to be taken care of is establishing the stability of the closed-loop system. This is one great property of integral control.

\subsection{A nonlinear positive integral controller}

We have shown before the properties of integral control in terms of output regulation and constant disturbance rejection which we would like to have to solve the control problem stated in Problem \ref{prob:1}. However, the constraints of Statement \emph{(St\ref{controlobj:1})} and Statement \emph{(St\ref{controlobj:2})} are not fulfilled by the previously discussed integral controller. It is easy to see that if $d_u$ or $d_y$ are too large, then the value of the control input at equilibrium will be negative. Even more importantly, the control input may take negative values even in the absence of disturbances during the transient phase of the system if the output exceeds by too much the set-point value. A way for resolving this issue was to introduce an ON-OFF static nonlinearity between the controller and the system; see e.g. \cite{Briat:12c,Briat:13h}. Similar ideas have also been recently considered in \cite{Guiver:15}. However, this is only possible in-silico as it seems rather difficult to implement this ON-OFF nonlinearity in-vivo. This motivates the introduction of an alternative integral controller with the additional constraint of being implementable in terms of chemical reactions. An integral controller structure that meets these requirements referred to as \emph{antithetic integral controller} has been proposed in \cite{Briat:15e} in the stochastic setting. We propose here a different approach based on the following nonlinear positive integral controller:
\begin{equation}\label{eq:ic2}
\begin{array}{rcl}
     u(t)&=&kv(t)\\
    \dot{v}(t)&=&\alpha v(t)(\mu-y(t))
\end{array}
\end{equation}
where $y$ is the controlled output, $v$ is the state of the integrator and $\mu$ is the desired set-point. The parameters $\alpha$ and $k$ are design parameters. We can immediately notice the similarities between the linear integrator \eqref{eq:intdsmdds} and the nonlinear positive integrator \eqref{eq:ic2}. Indeed, the latter can be viewed as a scaled version of the former where the right-hand side of \eqref{eq:intdsmdds} is rescaled by $\alpha I(t)$. The controller also has a nice population dynamics interpretation. Indeed, the reference reaction can be interpreted as a reproduction reaction and the measurement reaction as a predation reaction.


\paragraph{Reaction network implementation.} Interestingly, the model is polynomial and, therefore, can be described by the reaction network with mass-action kinetics given by
\begin{equation}
  \underbrace{\V{}\rarrow{\alpha\mu}\V{} + \V{}}_{\mbox{reference reaction}}\hspace{-3mm},\quad  \underbrace{\V{} + \Y{}\rarrow{\alpha}\Y{}}_{\mbox{measurement reaction}}\ \textnormal{and}\ \underbrace{\V{}\rarrow{k}\V{} + \sum_{i=1}^db_i\X{i}}_{\mbox{actuation reaction}}
\end{equation}
where $\V{}$ and $\Y{}$ are the species associated with the controller state and the output of the reaction network, respectively. The term $b_i$ is simply the $i$-th entry of the vector $B$ in \eqref{eq:mainsystL}.

\section{Nonlinear positive integral control of unimolecular reaction networks}

In this section, several results on the control of unimolecular reaction networks are obtained. The equilibrium points and their stability are first analyzed. Disturbance rejection and perfect adaptation properties as well as the controller metabolic cost are considered next.


\subsection{Assumptions}

We make here the following assumptions:
\begin{hyp}
  The matrix $A$ is Hurwitz stable; i.e. all its eigenvalues are located in the open left-half plane.
\end{hyp}
\begin{hyp}
  The triplet $(A,B,C)$ is output controllable; i.e. $CA^{-1}B\ne0$.
\end{hyp}
It is important to stress that the above assumptions are necessary conditions ensuring that our control problem is solvable. If, indeed, the matrix $A$ were unstable, then it would not be possible to find an integral controller of the form \eqref{eq:ic2}, except in some very particular cases. For instance, considering an unstable first-order linear system
\begin{equation}
\dot{x}(t)=x(t)+kv(t)
\end{equation}
where $v$ obeys \eqref{eq:ic2}. Since $x(t)\ge0$ and $v(t)\ge0$, then we can immediately see that it is not possible to stabilize the system around the value $\mu$ since we would need the control to be negative when $x(t)>\mu$. Note that if we consider instead the system
\begin{equation}
\dot{x}(t)=-x(t)+kv(t)
\end{equation}
then regulating $x(t)$ around $\mu$ becomes possible. The second assumption, finally, means that the system responds to an input. If $CA^{-1}B=0$, then this means that the transfer function $C(sI-A)^{-1}B$ is identically 0 and therefore that the output system does not respond to changes in the input.

\subsection{Equilibrium points}\label{sec:SM:equilibrium_points}

Let us start with the existence of equilibrium points:
\begin{proposition}
  The closed-loop system \eqref{eq:mainsystL}-\eqref{eq:ic2} admits two equilibrium points:
  \begin{enumerate}
    \item the first one is the zero-equilibrium point
    \begin{equation}\label{eq:trivial_eqpt}
      (x^*,v^*)=(0,0)
    \end{equation}
  \item the second one is the positive-equilibrium point
  \begin{equation}\label{eq:positive_eqpt}
     (x^*,v^*)=\left(\dfrac{A^{-1}B\mu}{CA^{-1}B},\dfrac{-\mu}{CA^{-1}Bk}\right)
  \end{equation}
  which exists provided that $\mu>0$.
  \end{enumerate}
\end{proposition}
Note that when $\mu=0$, the zero-equilibrium is retrieved from the positive one.

\subsection{Local stability of the equilibrium points}\label{sec:SM:local_stability}

\subsubsection{Local stability of the zero-equilibrium point}\label{sec:SM:local_stability_zero}

We establish now the stability/instability of the equilibrium points. Let us start with the trivial equilibrium point:
\begin{proposition}\label{prop:stab_zero_1}
The zero-equilibrium point \eqref{eq:trivial_eqpt} is unstable for all $\mu,\alpha,k>0$.
\end{proposition}
\begin{proof}
The linearized system about the zero-equilibrium point \eqref{eq:trivial_eqpt} is given by
  \begin{equation}
  \begin{bmatrix}
    \dot{\tilde{x}}(t)\\
    \dot{\tilde{v}}(t)
  \end{bmatrix}=\underbrace{\begin{bmatrix}
    A & Bk\\
    0 & \alpha \mu
  \end{bmatrix}}_{\mbox{$M_0$}} \begin{bmatrix}
    \tilde{x}(t)\\
    \tilde{v}(t)
  \end{bmatrix}
\end{equation}
where $\tilde x:=x-x^*$ and $\tilde v:=v-v^*$. Clearly, the eigenvalues of $M_0$ are given by $\lambda(A)\cup\{\alpha\mu\}$ where $\lambda(A)$ is the set of eigenvalues of $A$. The linearized systems is therefore unstable since $\alpha\mu>0$. The proof is complete.
\end{proof}

\subsubsection{Local stability of the positive equilibrium point}\label{sec:SM:local_stability_positive}

We address now the stability of the positive-equilibrium point:
\begin{proposition}\label{prop:stab_positive_1}
  Let $\mu>0$ be given. Then, there exists an $\bar{\alpha}=\bar{\alpha}(\mu)>0$ such that the positive equilibrium point \eqref{eq:positive_eqpt}
%
  is locally exponentially stable provided that $\alpha\in(0,\bar{\alpha})$ where
    \begin{equation}\label{eq:baralpha1}
      \bar{\alpha}:=\sup\left\{\nu>0:\ \begin{bmatrix}
        A & B\\
          \dfrac{\nu\mu C}{CA^{-1}B} & 0
      \end{bmatrix}\ \textnormal{Hurwitz stable}\right\}.
    \end{equation}
\end{proposition}
\begin{proof}
The linearized system about the positive equilibrium point  \eqref{eq:positive_eqpt} is given by
  \begin{equation}\label{eq:locstab1}
  \begin{bmatrix}
    \dot{\tilde{x}}(t)\\
    \dot{\tilde{v}}(t)
  \end{bmatrix}=\underbrace{\begin{bmatrix}
    A & Bk\\
    \dfrac{\alpha C\mu}{CA^{-1}Bk} & 0
  \end{bmatrix}}_{\mbox{$M_p$}}\begin{bmatrix}
    \tilde{x}(t)\\
    \tilde{v}(t)
  \end{bmatrix}
\end{equation}
where $\tilde x:=x-x^*$ and $\tilde v:=v-v^*$. By pre- and post-multiplying $M_p$ by $\diag(I_n,k)$ and $\diag(I_n,k^{-1})$, we get that the eigenvalues of $M_p$ are the same as those of
\begin{equation}
\tilde{M}_p:=\begin{bmatrix}
    A & B\\
    \dfrac{\alpha C\mu}{CA^{-1}B} & 0
  \end{bmatrix}
\end{equation}
and, therefore, that the gain $k>0$ does not affect the eigenvalues (the stability) of $M_p$. Hence, we only need to characterize the dependence of the eigenvalues on $\alpha,\mu>0$.

We prove first that $\alpha$ needs to be positive. To do so, we use a perturbation argument to show that the matrix $M_p$ is Hurwitz stable for some sufficiently small $\alpha>0$. When $\alpha=0$, the matrix $M_p$ has $n$ stable eigenvalues (those of $A$) and one eigenvalue at $\lambda_0=0$. The eigenvalue $\lambda_0$ has normalized left- and right-eigenvectors given by
\begin{equation}
\ u_\ell=\begin{bmatrix}
    0_{1\times n} & 1
  \end{bmatrix}\ \textnormal{and}\  u_r=\begin{bmatrix}
    -A^{-1}B \\ 1
  \end{bmatrix}
\end{equation}
respectively. The theory of perturbation of simple eigenvalues says that the eigenvalue $\lambda_0$ locally changes under the effect of a perturbation of $\eps$ on $\alpha$ around $\alpha=0$ according to the relation
\begin{equation}
  \lambda(\eps)=\lambda_0+\left(u_\ell\begin{bmatrix}
    0 & 0\\
    \dfrac{C\mu}{CA^{-1}B} & 0
  \end{bmatrix}u_r\right)\eps+o(\eps)
  \end{equation}
  where $o(\eps)$ is the little-o notation. We therefore get that $\lambda(\eps)=-\mu\eps+o(\eps)$ and therefore that the zero-eigenvalue moves to the open left-half plane for some sufficiently small $\eps>0$. This means that there exists a sufficiently small $\bar{\alpha}=\bar{\alpha}(\mu)>0$ such that for all $\alpha<\bar{\alpha}$, the matrix $M_p$ is Hurwitz stable.

  We know now that $\alpha$ must be positive. Its maximal value can be computed by simply looking at the value for $\alpha$ for which the matrix $M_p$ is not Hurwitz stable anymore. This then leads us to the characterization \eqref{eq:baralpha1}. The proof is complete.
\end{proof}

\begin{remark}
  Unlike in standard integral control for linear systems, the stability depends on the reference value $\mu$. It is therefore important here to consider the possible range of values for $\mu$ before designing the controller, i.e. choosing $\alpha$. This is a consequence of the fact that the controller is nonlinear.
\end{remark}

\textbf{Computing $\boldsymbol{\bar{\alpha}}$.} The upper-bound $\bar{\alpha}$ can be found by several ways. When the matrix is of small dimension, then the Routh-Hurwitz criterion can be considered. When the matrix is too large, numerical methods should be considered instead. A brute force approach could consider a bisection approach with, at each step, the evaluation of the eigenvalues of the matrix in \eqref{eq:locstab1}. A certainly more elegant approach relies on the calculation of the so-called \emph{stability crossings}. The idea is to evaluate the characteristic polynomial of the matrix in \eqref{eq:locstab1}, that we denote by $P(s,\alpha)$, where $(s,\alpha)\in\mathbb{C}\times\mathbb{R}_{\ge0}$, and to find a pair $(\omega,\alpha)\in\mathbb{R}_{\ge0}^2$ such that $P(j\omega,\alpha)=0$; i.e. find a value for $\alpha$ for which the matrix has a pair of conjugate eigenvalues on the imaginary axis. This leads us to the following result:
\begin{proposition}\label{prop:alpha_bar}
The value $\bar{\alpha}$ is given by
  \begin{equation}
  \bar{\alpha}=\dfrac{1}{\mu}\left\{\begin{array}{lcl}
    \dfrac{D_I(\omega_*)\omega_*}{N_R(\omega_*)}&\textnormal{if}&N_R(\omega_*)\ne0\\
    \dfrac{-D_R(\omega_*)\omega_*}{N_I(\omega_*)}&\textnormal{if}&N_I(\omega_*)\ne0.
  \end{array}\right.
\end{equation}
where  $\omega_*$ is the smallest positive solution to the polynomial equation
\begin{equation}\label{eq:polycrtic}
  Q(\omega):=N_I(\omega)D_I(\omega)+N_R(\omega)D_R(\omega)=0
\end{equation}
where
\begin{equation}\label{eq:Hn}
\dfrac{C(sI-A)^{-1}B}{-CA^{-1}B}=:H_n(j\omega)=:\dfrac{N_R(\omega)+jN_I(\omega)}{D_R(\omega)+jD_I(\omega)}
\end{equation}
 where the real polynomials $N_R(\omega)$, $D_R(\omega)$, $N_I(\omega)$ and $D_I(\omega)$ are obvious from the definition above.

When the polynomial equation \eqref{eq:polycrtic} has no positive solution, then $\bar{\alpha}=\infty$.
\end{proposition}
\begin{proof}
The matrix $M_p$ defined in \eqref{eq:locstab1} is Hurwitz stable if and only if
  \begin{equation}
    \det\begin{bmatrix}
    sI-A & -B\\
    -\dfrac{\alpha\mu C}{CA^{-1}B} & s
  \end{bmatrix}=0
  \end{equation}
  has its solutions in the open left-half plane. Applying then the Schur determinant formula yields
  \begin{equation}
    \det\begin{bmatrix}
    sI-A & -B\\
    -\dfrac{\alpha\mu C}{CA^{-1}B} & s
  \end{bmatrix}=\det(sI-A)\det\left(s-\dfrac{\alpha\mu C}{CA^{-1}B}(sI-A)^{-1}B\right).
  \end{equation}
  Since $A$ is Hurwitz, then the above polynomial is Hurwitz stable if and only if all the zeros of
   \begin{equation}\label{eq:P}
    P(s,\alpha):=s+\alpha\mu H_n(s)=0
  \end{equation}
  are located in the set $\{s\in\mathbb{C}:\ \Re[s]<0\}\times\mathbb{R}_{>0}$ where $H_n(s)$ is defined in \eqref{eq:Hn}. We now look for critical pairs $(\omega,\alpha)\in\mathbb{R}^2_{>0}$ such that $P(j\omega,\alpha)=0$.  Solving then for the problem $P(j\omega,\alpha)=0$, separating the real and imaginary parts and substituting for the value of $\alpha$ yield that there exists a critical pair $(\omega,\alpha)$ such that $P(j\omega,\alpha)=0$ if and only if there exists a positive solution $w_*$ to \eqref{eq:polycrtic}. When a such an $\omega^*>0$ is found, $\bar{\alpha}$ can be computing by solving for $\alpha$ in the real or the imaginary part equation of $P(j\omega,\alpha)=0$. When there is no such $w_*$, then $\alpha$ admits no upper-bound. The proof is complete.
\end{proof}

Interestingly there is a possibility of checking beforehand whether the polynomial \eqref{eq:polycrtic} has no positive root. This is stated in the following result:
\begin{proposition}\label{prop:arbitrary_alpha}
The following statements are equivalent:
\begin{enumerate}
  \item The function $C(sI-A)^{-1}B$ is weakly strictly positive real; i.e. $C(sI-A)^{-1}B$ has roots in the open left-half plane and $\Re[C(j\omega I-A)^{-1}B]>0$ for all $\omega\in\mathbb{R}$.
  \item The positive-equilibrium point  \eqref{eq:positive_eqpt} is locally exponentially stable for any $\mu,\alpha,k>0$.
\end{enumerate}
\end{proposition}
\begin{proof}
Note first that $P(j\omega,\alpha)=0$ for some $(\omega,\alpha)\in\mathbb{R}_{>0}^2$ if and only if
\begin{equation}\label{eq:mag}
  \omega = \alpha\mu |H_n(j\omega)|
\end{equation}
and
\begin{equation}\label{eq:phase}
  -\dfrac{\pi}{2} = \arg(H_n(j\omega)).
\end{equation}
Clearly, if there is an $\omega$ such that the equality \eqref{eq:phase} holds, then \eqref{eq:mag} can be automatically satisfied by appropriately choosing $\alpha$. Such an $\omega$ cannot be 0 since $H(0)\ne0$ and cannot be $\infty$ since $H(\infty)=0$. Therefore, the nonzero equilibrium point is locally exponentially stable for any $\mu,\alpha,k>0$ if and only if the phase of $H(j\omega)$ is bounded away from $-\pi/2$ for all $\omega\in\mathbb{R}$. Which is the definition of a weakly strictly positive real function. The proof is complete.
\end{proof}

\begin{remark}
Establishing whether the transfer function $C(sI-A)^{-1}B$ is weakly strictly positive real is not easy in the general case. However, the transfer function can be proved to strictly positive real (which is slightly stronger) by checking the existence of symmetric positive definite matrices $P,Q\in\mathbb{R}^{d\times d}$ such that we have
\begin{equation}
    A^TP+PA=-Q\ \textnormal{and}\ PB=C^T.
\end{equation}
\end{remark}

\begin{example}
We now give some examples:
\begin{enumerate}
    \item The function $\dfrac{1}{(s+\alpha)}$, $\alpha>0$, is strictly positive real (it is hence also weakly strictly positive real).
    \item The function $\dfrac{1}{(s+\alpha)(s+\beta)}$, $\alpha,\beta>0$, is not positive real.
    \item The function $\dfrac{s+\alpha+\beta}{(s+\alpha)(s+\beta)}$, $\alpha,\beta>0$, is weakly strictly positive real but not strictly positive real.
\end{enumerate}
\end{example}

\textbf{Choosing the best $\boldsymbol{\alpha}$.} How $\alpha$ should be chosen? A relevant measure is that of \emph{spectral abscissa} defined for a square matrix $M$ as
\begin{equation}
  \chi(M):=\max\{\Re[\theta]:\ \theta\in\lambda(M)\}
\end{equation}
where $\lambda(M)$ denotes the spectrum of the matrix $M$. The spectral abscissa measures the rate of convergence of the solutions to the equilibrium point for trajectories that are sufficiently close to the equilibrium point. Based on this, $\alpha$ should be chosen as the value in $(0,\bar{\alpha})$ such that
\begin{equation}
\chi\left(\begin{bmatrix}
  A & B\\
  \dfrac{\alpha\mu}{CA^{-1}B}C & 0
\end{bmatrix}\right)
\end{equation}
is minimum. To compute this value, several methods can be applied. A brute force method consists of performing a gridding of the interval $(0,\bar{\alpha})$ and evaluating the spectral abscissa at each point. Then $\alpha$ can be picked as the one that minimizes the spectral abscissa over the grid. Methods based on nonsmooth optimization may also be used; see e.g. \cite{Burke:06b,Vanbiervliet:09}.

\subsubsection{Convergence of the time-averages}\label{sec:SM:time_averages}

It seems interesting to study the case where the system admits an attractive oscillatory solution (limit cycle) or non-converging bounded trajectories. In such a case, it is clear that the control problem is not solved since the unique positive steady-state value is not asymptotically stable and, therefore, set-point tracking is not achieved. It seems then natural to relax the control problem into the following one:
\begin{problem}
  Let $\mu>0$. Find a controller such that
  \begin{enumerate}
   \item the control input $u$ is nonnegative;
    \item the closed-loop system has bounded trajectories (at least for some set of initial conditions)
    \item the time-average of the output $y$ tracks the reference $\mu>0$; i.e.
    \begin{equation}
          \lim_{t\to\infty}\dfrac{1}{t}\int_0^ty(s)ds=\mu
    \end{equation}
    \item Constant disturbances acting on the input, on the state of the system and (possibly) on the output are rejected (at the time-average level).
    \end{enumerate}
\end{problem}

We then have the following proposition:
\begin{proposition}\label{prop:time_averages}
  Assume that $\mu$ is positive and that the solution of the system does not increase without bound and there exist positive constants $a<b$ such that $a<x_i(t)<b$ and $a<v(t)<b$, $i=1,\ldots,n$, for all $t\ge0$, then
  \begin{equation}
    \lim_{t\to\infty}\dfrac{1}{t}\int_0^tx(s)ds=x^*\ \textnormal{and}\ \lim_{t\to\infty}\dfrac{1}{t}\int_0^tv(s)ds=v^*.
  \end{equation}
%
%
\end{proposition}
\begin{proof}
  Since $\mu>0$, then $(x^*,v^*)$ is the only equilibrium in the interior of the positive orthant. Note also that $v(0)>0$, then $v(t)>0$ for all $t\ge0$. We then have that
  \begin{equation}
    \dfrac{x(t)-x(0)}{t}=\dfrac{1}{t}\int_0^t[Ax(s)+Bkv(s)]ds
  \end{equation}
  and
  \begin{equation}
    \dfrac{\log(v(t))-\log(v(0))}{t}=\dfrac{1}{t}\int_0^t[\alpha(\mu-Cx(s))]ds.
  \end{equation}
  Letting
\begin{equation}
  \bar{x}(t):=\dfrac{1}{t}\int_0^tx(s)ds\ \textnormal{and}\   \bar{v}(t):=\dfrac{1}{t}\int_0^tv(s)ds
\end{equation}
yields the expressions
  \begin{equation}\label{eq:dksldsldq111}
    \dfrac{x(t)-x(0)}{t}=A\bar{x}(t)+Bk\bar{v}(t)\ \textnormal{and}\   \dfrac{\log(v(t))-\log(x(0))}{t}=\alpha(\mu-C\bar{x}(t)).
  \end{equation}
Since $a<\bar{x}_i(t)<b$, $i=1,\ldots,d$, and $a<\bar{v}(t)<b$ for all $t>0$, then for any sequence $\{t_k\}$ such that ${t_k\to\infty}$ as $k\to\infty$, and for each $i=1,\ldots,d$, the sequence $\{\bar{x}_i(t_k)\}$ admits a convergent subsequence. The same holds for the sequence $\{\bar{v}(t_k)\}$. Hence, we can define a subsequence $\{t_{\sigma(k)}\}$ such that $(\bar{x}(t_{\sigma(k)}),\bar{v}(t_{\sigma(k)}))$ converges to some limit $(\bar{x}^*,\bar{v}^*)$. Using now the fact that sequences $x_i(t_{\sigma(k)})-\bar{x}_i(0)$ and $\log(v(t_{\sigma(k)}))-\log(v(0))$ are bounded, we can take the limit in \eqref{eq:dksldsldq111} to get
\begin{equation}
  0=A\bar{x}^*+Bk\bar{v}^*\ \textnormal{and}\   0=\alpha(\mu-C\bar{x}^*).
\end{equation}
Then then implies that the point $(\bar{x}^*,\bar{v}^*)$ is an equilibrium point of the system and since $\bar{x}_1^*,\ldots,\bar{x}^*_d,\bar{v}^*>a$, then it belongs to the interior of the positive orthant meaning that $\bar{x}^*=x^*$ and $\bar{v}^*=v^*$. The proof is complete.
\end{proof}

The above notably ensures that even if the trajectories are periodic (e.g. follow a limit cycle), then the control problem is still solved, but a weaker sense.

\subsection{Disturbance rejection and perfect adaptation properties}\label{sec:SM:perfect_adaptation}

Let us consider now the following disturbed system
\begin{equation}\label{eq:mainsystLdist}
  \begin{array}{lcl}
    \dot{x}(t)&=&Ax(t)+Bu(t)+Ed\\
    y(t)&=&Cx(t)\\
    x(0)&=&x_0
  \end{array}
\end{equation}
where $x\in\mathbb{R}^n$ is the state of the system, $u\in\mathbb{R}$ is the control input, $d\in\mathbb{R}_{\ge0}^p$ is the constant input disturbances and $y\in\mathbb{R}$ is the measured output that has to be controlled. As before, we assume that the matrix $A$ is Metzler and Hurwitz stable, and that the matrices $B,C,E$ are nonnegative. Therefore, the system is internally positive, which means that for any $x_0\ge0$, any $d\ge0$ and any $u(t)\ge0$, we have that $x(t)\ge0$ and $y(t)\ge0$ for all $t\ge0$. Alternatively, this model can be viewed as a deterministic reaction network involving zeroth and first order reactions with the matrix $A=S\Lambda$ and $Ed=S\lambda_0$ where $d$ is defined as the vector that contains the distinct nonzero entries of $\lambda_0$. As before, the matrix $B$ can be chosen so that $u$ is the birth rate of some species whereas $C$ is simply chosen so as to measure one particular species; see e.g. \cite{Briat:15e} for a similar setup in the stochastic setting.

It is convenient here to introduce here the following set of admissible constant input disturbances given by
\begin{equation}\label{eq:calD}
  \mathcal{D}_\mu:=\left\{d\in\mathbb{R}^{p}_{\ge0}:\ \mu+CA^{-1}Ed>0\right\}.
\end{equation}
Basically, this set contains the admissible disturbance amplitudes for which the output can track the reference $\mu$. Indeed, if the disturbance is such that the output of the system for $u=0$ is larger than the reference, then the reference cannot be tracked (because this would require a negative input). Note that if $-CA^{-1}E=0$, i.e. the disturbances do impact the output, then the disturbances can then be arbitrarily large and $\mathcal{D}_\mu=\mathbb{R}^{p}_{\ge0}$.

We now analyze the stability of the equilibrium points when disturbances are present. Let us start with the perturbed zero-equilibrium point and prove that it is structurally unstable, as in the disturbance-free case.
\begin{proposition}\label{prop:dist_zero}
  Assume that $d\in\mathcal{D}_\mu$, then the perturbed zero-equilibrium point
  \begin{equation}
    \left(x^*,v^*\right)=\left(-A^{-1}Ed,0\right)
  \end{equation}
  is unstable for all $\alpha,k,\mu>0$.
\end{proposition}
\begin{proof}
The proof is identical to the one of Proposition \ref{prop:stab_zero_1}.
\end{proof}

We now consider the perturbed positive-equilibrium point:
\begin{proposition}\label{prop:dist_positive}
  Let $\mu>0$ be given and assume that $d\in\mathcal{D}_\mu$, then there exists a $\bar{\alpha}_d=\bar{\alpha}_d(\mu)>0$ such that the perturbed positive-equilibrium point
  \begin{equation}
(x^*,v^*)=\left(A^{-1}\left(\dfrac{B(\mu+CA^{-1}Ed)}{CA^{-1}B}-Ed\right),-\dfrac{\mu+CA^{-1}Ed}{CA^{-1}Bk}\right)
  \end{equation}
  is locally exponentially stable for all $k>0$ and all $\alpha\in(0,\bar{\alpha}_d)$ where
  \begin{equation}
     \bar{\alpha}_d=\dfrac{\mu}{\mu+CA^{-1}Ed}\bar{\alpha}\ge\bar{\alpha}
  \end{equation}
  where $\bar{\alpha}$ is defined in Proposition \ref{prop:alpha_bar}.
\end{proposition}
\begin{proof}
 As we have seen before, the equilibrium value $x^*$ of the state does not change the expression of the linearized system. Only the equilibrium value $v^*$ of the state of the nonlinear integrator matters. Noting that $v^*(d)=v^*(0)+\Delta_vd$ where
 \begin{equation}
  \begin{array}{rclcrcl}
    v^*(0)&=&\dfrac{-\mu}{CA^{-1}Bk}>0&\textnormal{and}&\Delta_v&=&-\dfrac{CA^{-1}E}{CA^{-1}Bk}\le0
  \end{array}
\end{equation}
allows us to conclude that $v^*(d)\le v^*(0)$ for all $d\in\mathcal{D}_\mu$. Interestingly, the local linear system becomes in this case
\begin{equation}
\begin{bmatrix}
  \dot{\tilde{x}}(t)\\
  \dot{\tilde{v}}(t)
\end{bmatrix}  =\begin{bmatrix}
    A & Bk\\
    -\alpha v^*(d)C & 0
  \end{bmatrix}\begin{bmatrix}
  \tilde{x}(t)\\
  \tilde{v}(t)
\end{bmatrix}
\end{equation}
and assuming that the system is not stable for an arbitrarily large $\alpha$ (see Proposition \ref{prop:arbitrary_alpha}), then the above system is asymptotically stable if and only if the zeros of polynomial
   \begin{equation}\label{eq:P2}
    P(s,\alpha)=s+\alpha (\mu+CA^{-1}Ed)H_n(s)
  \end{equation}
  are located in the set $\{s\in\mathbb{C}:\ \Re[s]<0\}\times\mathbb{R}_{>0}$ where $H_n(s)$ is defined in \eqref{eq:Hn}. If we, therefore, let $\bar\alpha$ to be the maximal admissible value for $\alpha$ when $d=0$ and $\bar\alpha_d$ when the value is $d\ne0$, then we have that
\begin{equation}
  \bar{\alpha}_d:=\dfrac{\mu}{\mu+CA^{-1}Ed}\bar{\alpha}\ge\bar{\alpha}.
\end{equation}
Therefore, if $\alpha\in(0,\bar{\alpha})$, then the system is locally exponentially stable for any $d\in\mathcal{D}_\mu$. The proof is complete.
\end{proof}

\begin{remark}
Quite interestingly, the function $v^*(d)$ is a decreasing function of the disturbance level $d$. The disturbance has then the effect of enlarging the admissible region for the parameter $\alpha$; i.e. it increases $\bar\alpha$.
\end{remark}


%
%


\subsection{Robustness properties}

Closed-loop asymptotic stability and set-point tracking is a robust property meaning that it will preserved under sufficiently small perturbations of the parameters of the reaction network and the controller. In this regard, when reaction rates change value sporadically enough and stay in the stability region, then the system will perfectly adapt to these changes.

\subsection{Metabolic cost and power consumption}\label{sec:SM:metabolic_cost}

An interesting way to evaluate the metabolic cost induced by the introduction of the additional reactions is to evaluate the long-run time-averages of the propensity functions associated with the controller reactions. Clearly, three reactions are involved, namely, the reference reaction, the measurement reaction and the actuation reaction. to which we associate the unitary metabolic costs $\kappa_r$, $\kappa_m$ and $\kappa_a$, respectively. The total energy consumed until time $t$ is given by
\begin{equation}
  E(t)=\kappa_r\int_0^t\alpha\mu v(s)ds+\kappa_m\int_0^t\alpha v(s)y(s)ds+\kappa_a\int_0^tk v(s)ds.
\end{equation}
This quantity grows without bound irrespectively of the stability of the process. To circumvent this problem, we will consider the instantaneous power defined as:
\begin{equation}
  P(t)=\dfrac{dE(t)}{dt}
\end{equation}
and, more especially, its stationary value $P^*=\lim_{t\to\infty}P(t)$ when it exists. We then have the following result:
\begin{proposition}
Assume that the positive equilibrium point of the closed-loop network is asymptotically stable. Then, the stationary power consumption of the controller network is given by
\begin{equation}
  P^*=\dfrac{\mu}{-CA^{-1}B}\left(\dfrac{\alpha\mu(\kappa_r+\kappa_m)}{k}+\kappa_a\right).
\end{equation}
\end{proposition}
\begin{proof}
The proof follows from simple substitutions.
\end{proof}

Very interestingly, we can see that by increasing the gain $k>0$, we can reduce the equilibrium power consumption without risking deteriorating the stability of the system. Also note we have the immediate lower-bound $\mu\kappa_a$ on the equilibrium power consumption which corresponds to the power consumption of the actuation reaction.

\begin{remark}
Note that when the closed-loop network is not asymptotically stable, then $P^*$ may not exist. To overcome this situation, the idea would to consider the average value of $P$ given by
\begin{equation}
  \lim_{t\to\infty}\dfrac{1}{t}\int_0^tP(s)ds
\end{equation}
when the limit exists. This can be proved using similar arguments as for the proof of Proposition \ref{prop:time_averages}.
\end{remark}

\section{Examples}

\subsection{Birth-death process}

Let us consider the following reaction network:
\begin{equation}
  \phib\rarrow{k_b}\X{},\quad   \X{}\rarrow{\gamma}\phib
\end{equation}
where $\X{}$ is, in this case, both the actuated and the measured species. Assuming mass-action kinetics, the above network can be described by the following ordinary linear differential equation:
  \begin{equation}
      \dot{x}(t)=u(t)-\gamma x(t)
  \end{equation}
  where $x$ is the concentration of the molecular species $\X{}$. The control input is chosen to be the birth rate $k_b$. The transfer function of this system is given by $1/(s+\gamma)$ and hence the closed-loop network will be stable for all $\mu,\alpha>0$. Consequently, the closed-loop network will exhibit output tracking and perfect adaptation. Moreover, the metabolic cost is given, in this case, by
  \begin{equation}
  P^*=\dfrac{\mu}{\gamma}\left(\dfrac{\alpha\mu(\kappa_r+\kappa_m)}{k}+\kappa_a\right).
  \end{equation}

%
%

\subsection{Gene expression with protein maturation}\label{sec:SM:gene_expression}

\textbf{Theoretical results.} Let us consider the following reaction network:
\begin{equation}
  \phib\rarrow{k_m}\M{},\quad
  \M{}\rarrow{\gamma_m}\phib,\quad
  \M{}\rarrow{k_p}\M{} + \P{},\quad
  \P{}\rarrow{\gamma_p}\phib,\quad
  \P{}\rarrow{k_p}\Q{},\quad
  \Q{}\rarrow{\gamma_q}\phib
\end{equation}
where $\M{}$, $\P{}$ and $\Q{}$ represent the mRNA, protein and mature protein species, respectively. Assuming mass-action kinetics, the above network can be described by the following set of ordinary linear differential equations:
  \begin{equation}\label{eq:example_th}
    \begin{array}{lcl}
      \dot{m}(t)&=&k_m-\gamma_mm(t)\\
      \dot{p}(t)&=&k_pm(t)-(\gamma_p+k_q)p(t)\\
      \dot{q}(t)&=&k_qp(t)-\gamma_qq(t)
    \end{array}
  \end{equation}
  where $m,p,q$ are the concentrations of mRNA molecules, protein molecules and mature protein molecules, respectively. We choose the transcription factor $k_m$ to be the control input and the measured output is the concentration of mature protein molecules. Then the model is of the form \eqref{eq:mainsystL} with the matrices:
\begin{equation}
  A = \begin{bmatrix}
    -\gamma_m    &0&    0\\
    k_p  & -(\gamma_p+k_q) & 0\\
    0      &   k_q  & -\gamma_q
  \end{bmatrix},\ B=\begin{bmatrix}
    1\\0\\0
  \end{bmatrix},\ C=\begin{bmatrix}
    0\\0\\1
  \end{bmatrix}^T.
\end{equation}
The corresponding transfer function is given by
\begin{equation}
  H_n(s)=\dfrac{\gamma_m\gamma_q(\gamma_p+k_q)}{s^3+(\gamma_m + \gamma_p + \gamma_q + k_q)s^2+(\gamma_m\gamma_q + (\gamma_m + \gamma_q)(\gamma_p + \gamma_q))s+\gamma_m\gamma_q(\gamma_p+k_q)}.
\end{equation}
We can immediately see that there will be an $\omega>0$ such that $\arg(H_n(j\omega))=-\pi/2$ since the degree of the numerator is 0 and the degree of the denominator is 3. Hence, $\bar{\alpha}$ will be finite. To find it, we need to build the polynomial $Q(\omega)$ as described in \eqref{eq:polycrtic} and we get
%
%
  \begin{equation}
    Q(\omega)=\gamma_m\gamma_q(\gamma_p+k_q)-\omega^2(\gamma_m + \gamma_p + \gamma_q + k_q)
  \end{equation}
  from which we can conclude that $Q(\omega_*)=0$ with $\omega_*=\sqrt{\dfrac{\gamma_m\gamma_q(\gamma_p+k_q)}{\gamma_m + \gamma_p + \gamma_q + k_q}}$. Since $N_I(\omega)=0$ identically, we find that
  \begin{equation}
    \bar{\alpha}=\dfrac{(\gamma_m + \gamma_q)(\gamma_m + \gamma_p + k_q)(\gamma_p + \gamma_q + k_q)}{\mu(\gamma_m + \gamma_p + \gamma_q + \gamma_q)^2}.
  \end{equation}
  As a validation of the result, we can evaluate $H_n(j\omega_*)$ and we get
  \begin{equation}
    H_n(j\omega_*)=-j\dfrac{(\gamma_m \gamma_q (\gamma_p + k_q))^{1/2}(\gamma_m + \gamma_p + \gamma_q + k_q)^{3/2}}{(\gamma_m + \gamma_q) (\gamma_m \gamma_p + 2 \gamma_m \gamma_q + 2 \gamma_p \gamma_q + \gamma_p k_q + \gamma_q k_q + \gamma_p^2 + \gamma_q^2)}
  \end{equation}
  from which we can see that we have $\arg (H_n(j\omega_*))=-\pi/2$.

\subsection{A bimolecular dimerization process}\label{sec:SM:dimerization}

Even if the theory has been developed for unimolecular reaction networks only, the proposed integral controller still works on more general reaction networks such as bimolecular reaction networks. Let us consider for instance the following reversible dimerization process also considered in \cite{Briat:13h}:
\begin{equation}
    \X{1}\rarrow{\gamma_1}\phib,\  \X{1} + \X{1}\rarrow{k_{12}}\X{2},\  \X{2}\rarrow{k_{21}}\X{1} + \X{1},\
    \X{2}\rarrow{\gamma_2}\phib
\end{equation}
where $\X{1}$ is the actuation species, i.e. $\V{}\rarrow{k}\V{}+\X{1}$, and $\X{2}$ is the measured species, i.e. $\Y{}\equiv\X{2}$. The model of the closed-loop network is therefore given by
\begin{equation}
  \begin{array}{rcl}
    \dot{x}_1(t)&=&-\gamma_1x_1(t)-2k_{12}x_1(t)^2+2k_{21}x_2(t)+kv(t)\\
    \dot{x}_2(t)&=&-\gamma_2x_2(t)+k_{12}x_1(t)^2-k_{21}x_2(t)\\
    \dot{v}(t)&=&\alpha v(t)(\mu-x_2(t))
  \end{array}
\end{equation}
where $x_1,x_2$ and $v$ are the molecular concentrations associated with the molecular species $\X{1}$, $\X{2}$ and $\V{}$, respectively. As for unimolecular networks, two equilibrium coexist: the zero equilibrium point $(0,0,0)$ and the positive equilibrium point
\begin{equation}
  (x_1^*,x_2^*,v^*)=\left(\sqrt{\dfrac{\mu(\gamma_2+k_{21})}{k_{12}}},\mu,\dfrac{1}{k}\left(\gamma_1\sqrt{\dfrac{\mu(\gamma_2+k_{21})}{k_{12}}}+\gamma_2\mu\right)\right).
\end{equation}

 The zero equilibrium point is unstable since the Jacobian matrix
\begin{equation}
  \begin{bmatrix}
    -\gamma_1 & 2k_{21} & k\\
    0 & -\gamma_2 & 0\\
    0 & 0 & \alpha\mu
  \end{bmatrix}
\end{equation}
has a positive eigenvalue. On the other hand, the positive equilibrium point is locally asymptotically stable if the matrix
\begin{equation}
  \begin{bmatrix}
    -\gamma_1-4k_{12}x_1^* & 2k_{21} & k\\
    2k_{12}x_1^* & -(\gamma_2+k_{21}) & 0\\
    0 & -\alpha v^* & 0
  \end{bmatrix}
\end{equation}
is Hurwitz stable. By noting that the $2\times 2$ left-upper matrix is Metzler and Hurwitz stable (negative trace and positive determinant), then the results developed in the unimolecular case apply and we can conclude that there exists an $\bar{\alpha}(\mu)>0$ such that for all $\alpha\in(0,\bar{\alpha}(\mu))$ the matrix is Hurwitz stable. The closed-loop network will also exhibit tracking and perfect adaptation. Finally, the metabolic cost can be seen to be equal to
\begin{equation}
  P^*=u^*\left(\dfrac{\alpha\mu(\kappa_r+\kappa_m)}{k}+\kappa_a\right)>\alpha \kappa_au^*
\end{equation}
and
\begin{equation}
  u^*=\gamma_1\sqrt{\dfrac{\mu(\gamma_2+k_{21})}{k_{12}}}+\gamma_2\mu.
\end{equation}
Even though the formula is different from the unimolecular case, exactly the same conclusions can be drawn:
\begin{itemize}
  \item the metabolic cost is an increasing function of the set-point $\mu$.
  \item the metabolic cost can be reduced by increasing the gain $k>0$.
  \item the metabolic cost cannot be arbitrarily reduced due to the presence of a lower bound.
\end{itemize}

\begin{figure}
    \centering
 \includegraphics[width=7in]{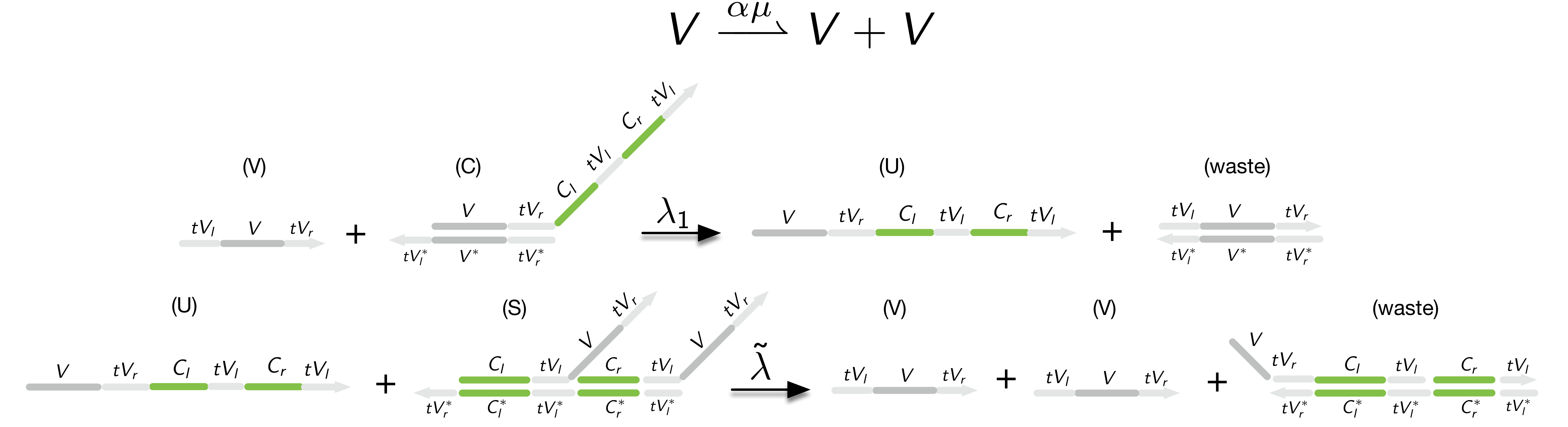}
    \caption{Detailed chart of the DNA implementation of the autocatalytic reference reaction. }\label{fig:SI_DNA}
\end{figure}




\end{document}